\documentclass[a4paper]{article}

\usepackage{amsthm,mathtools,amssymb,mathrsfs,bbold,natbib}
\usepackage[shortlabels]{enumitem}
\RequirePackage[colorlinks,citecolor=blue,urlcolor=blue]{hyperref}
\usepackage{geometry}
\makeindex             % used for the subject index
\newtheorem{Lem}{Lemma}[section]

\newtheorem{Thm}[Lem]{Theorem}

\newtheorem{Conj}[Lem]{Conjecture}

\newtheorem{Def}[Lem]{Definition}

\newtheorem{Rm}[Lem]{Remark}
\newtheorem{Notation}[Lem]{Notation}
\DeclareMathOperator*{\argmin}{\arg\!\min}
\DeclareMathOperator*{\argmax}{\arg\!\max}

\DeclareMathOperator*{\Mod}{\;\text{mod}\;}

\begin{document}

\title{The Long Time Limit of Diffusion Means}
\author{Till D\"usberg\footnote{Institute for Mathematical Stochastics, University of G\"ottingen}{ } and Benjamin Eltzner\footnote{Max Planck Institute for Multidisciplinary Sciences, G\"ottingen, \texttt{benjamin.eltzner@mpinat.mpg.de}}}
%\institute{Till D\"usberg \at Institute for Mathematical Stochastics, Goldschmidtstraße 7, 37077 G\"ottingen, \email{till.duesberg@stud.uni-goettingen.de}
%\and Benjamin Eltzner \at Max Planck Institute for Multidisciplinary Sciences, Am Fassberg 11, 37077 G\"ottingen \email{benjamin.eltzner@mpinat.mpg.de}}
\maketitle

\abstract{ In statistics on manifolds, the notion of the mean of a probability distribution becomes more involved than in a linear space. Several location statistics have been proposed, which reduce to the ordinary mean in Euclidean space. A relatively new family of contenders in this field are Diffusion Means, which are a one parameter family of location statistics modeled as initial points of isotropic diffusion with the diffusion time as parameter. It is natural to consider limit cases of the diffusion time parameter and it turns out that for short times the diffusion mean set approaches the intrinsic mean set. For long diffusion times, the limit is less obvious but for spheres of arbitrary dimension the diffusion mean set has been shown to converge to the extrinsic mean set. Here, we extend this result to the real projective spaces in their unique smooth isometric embedding into a linear space. We conjecture that the long time limit is always given by the extrinsic mean in the isometric embedding for connected compact symmetric spaces with unique isometric embedding.}

\section{Introduction}
\label{sec:intro}
One of the most basic problems facing attempts to establish statistical methods on non-Euclidean spaces is the fact that the mean cannot be easily generalized. A very general and early approach by Fr\'echet, see \cite{frechet_les_1948}, defines the set of means as the set of global minimizers of the expected squared distance on any metric space. Since non-Euclidean spaces admit a multitude of sensible metrics, this gives rise to a variety of potential means.

On smooth Riemannian manifolds, the metric tensor defines a canonical metric function, which gives rise to the so-called \textit{intrinsic mean set}. If a smooth Riemannian manifold is embedded into a Euclidean space, one can also consider the chordal distance, i.e. the distance along a straight line in the embedding space to define the so-called \textit{extrinsic mean set}, first discussed by \cite{HL96} and \cite{bhattacharya_large_2003}. Obviously, the extrinsic mean set depends strongly on the chosen embedding. More recently, \cite{EHHS2023} have introduced the so-called \textit{diffusion mean}, a one parameter set of mean candidates which are defined by the maximum likelihood parameter of the center of an isotropic Brownian motion, where the diffusion time is the free parameter. In contrast to the Euclidean case, the solution set is not independent of diffusion time and can contain more than one element.

In \cite{EHHS2023}, the two limits of diffusion time $t \to 0$ and $t \to \infty$ were considered and for the former case it was found that the diffusion mean set converges to a subset of the intrinsic mean set under weak conditions. The latter limit was only considered for the circle and spheres of arbitrary dimension in the respective canonical embeddings. In these cases, the limit for $t \to \infty$ was shown to be a subset of the extrinsic mean set.

Since the extrinsic mean set depends on the embedding, it is not immediately clear if and how the result for $t \to \infty$ can be generalized to other manifolds. In the present paper, we consider diffusion means on real projective spaces of arbitrary dimension which are smoothly and isometrically embedded into Euclidean space. Due to the Nash Embedding Theorem, such an embedding exists and for real projective spaces, \cite{zhang} give an explicit form of such an embedding and cite the result by \cite{DoCarmoWallach1971} from which it follows that the embedding is unique up to isometries of Euclidean space.

In this setting, we generalize the result found for spheres and we conjecture that for all compact connected symmetric spaces with a smooth isometric embedding that is unique up to isometries of Euclidean space the diffusion mean set converges to a subset of the extrinsic mean set with respect to the smooth isometric embedding for $t \to \infty$.

The results presented in this paper are for the most part an abridged version of the results in \cite{Duesberg2024_master}. Theorems and proofs from this thesis are highlighted by citation below and are reproduced here almost verbatim.

\section{Preliminaries}
\label{sec:prelim}

We begin by introducing the relevant statistical and geometric objects. First, we introduce real projective spaces with radius $r > 0$ following \cite{Duesberg2024_master}.

\begin{Def}[Real projective space]\label{def:rp}\cite{lee_smooth}[p.6]\cite{lee_top}[p.101]
Let $m \in \mathbb{N}$. The $m$-dimensional \emph{real projective space} is the set of one dimensional linear subsets of $\mathbb{R}^{m+1}$. For an $r > 0$, it can be characterised as the $m$-dimensional \emph{real projective space with radius $r$} by
$$\mathbb{RP}^{m}(r) := \{[x]_{\sim}: x \in \mathbb{S}^{m}(r)\},$$
where $x \sim y :\Leftrightarrow x = y \lor x = -y $.
By convention, $\mathbb{RP}^{m} := \mathbb{RP}^{m}(1)$.
Denote by $\pi_{r}$ the quotient map, which is defined by 
$$ \pi_{r}:\mathbb{S}^{m}(r) \to \mathbb{RP}^{m}(r), \;\pi_{r}(x) := [x]_{\sim}.$$
By convention, $\pi := \pi_{1}$.
\end{Def}

In all of the following, we restrict to the case $m \ge 2$ for simplicity. The case $m=1$ is equivalent to the circle, which has been considered in \cite{EHHS2023}. Next we introduce the heat kernel and the diffusion means.

\begin{Def}[Heat kernel and diffusion means]\phantom{ }\\ \cite{EHHS2023}
    \begin{enumerate}
        \item A Riemannian manifold $(\mathcal{M}, g)$ is called \emph{stochastically complete} if there exists a minimal solution $p$ of the heat equation $\frac{d}{dt}p(x,y,t) = \frac{1}{2}\Delta_{g,x} p(x,y,t)$ satisfying $\int_{\mathcal{M}} p(x,y,t)dy = 1$ for all $x\in\mathcal{M}$ and $t>0$. Here, $\Delta_{g,x}$ denotes the Laplace-Beltrami operator with respect to the metric tensor $g$ acting on a function of $x$. We say that $p$ is the \emph{minimal heat kernel} of $\mathcal{M}$.
        \item With the probability space $(\Omega, \mathcal{F},\mathbb{P})$, let $X$ be a random variable on $\mathcal{M}$ and fix $t>0$. The \emph{diffusion $t$-mean set} $E_t(X)$ of $X$ is the set of global minima of the log-likelihood function $L_t = \mathbb{E}[-\ln(p(X,y,t))]$, i.e. 
	\begin{equation*}
	E_t(X) = \argmin_{y\in \mathcal{M}}\mathbb{E}[-\ln(p(X,y,t))].
	\end{equation*}
	If $E_t(X)$ contains a single point $\mu_t$, we say that $\mu_t$ is the \emph{diffusion $t$-mean} of $X$. 
    \end{enumerate}
\end{Def}

For connected compact Riemannian manifolds, the heat kernel can be expressed as an asymptotic series, the so-called heat kernel expansion introduced below.

\begin{Thm}\cite{chavel}[p.139f]\cite{grigoryan}[p.277]\label{thm:heat_expansion}
Let $(\mathcal{M},g)$ be a connected compact Riemannian mani\-fold. Then the following statements are true:
\begin{enumerate}
\item[1)]  The set of eigenvalues of $\frac{1}{2}\Delta_{g}$ is countably infinite, and can therefore be denoted by $(\lambda_{l})_{l=0}^{\infty}$, where each $\lambda_{l} \in \mathbb{R}$ is a distinct eigenvalue of $\frac{1}{2}\Delta_{g}$ with multiplicity $N_{l} \in \mathbb{N}$.
\item[2)] For all $l \in \mathbb{N}_{0}$, there exist $N_{l}$ smooth eigenfunctions $\varphi_{l,1},\dots,\varphi_{l,N_{l}}$ of $\frac{1}{2}\Delta_{g}$ to the distinct eigenvalue $\lambda_{l}$, such that the heat kernel $p$ of $\mathcal{M}$ takes the form
$$p(x,y,t) = \sum_{l=0}^{\infty}e^{\lambda_{l}t} \sum_{k=1}^{N_{l}} \varphi_{l,k}(x)\varphi_{l,k}(y),$$ 
and $\left((\varphi_{l,k})_{k=1}^{N_{l}}\right)_{l=0}^{\infty}$ forms an orthonormal basis of the space $L^{2}(\mathcal{M},\mathbb{R})$.
\item[3)] Any orthonormal basis of $L^{2}(\mathcal{M},\mathbb{R})$ comprising respectively $N_{l}$ eigenfunctions of $\frac{1}{2}\Delta_{g}$ to each of the distinct eigenvalues $\lambda_{l}$ gives rise to the expansion of $p$ above, which will be referred to as \emph{heat kernel expansion}. 
\item[4)] The heat kernel expansion converges absolutely for all $t > 0$.
\end{enumerate}
\end{Thm}

For the derivation of the heat kernel expansion in the real projective space, we use a somewhat coarse grained argument compared to \cite{Duesberg2024_master} to simplify the presentation.

\begin{Thm}[Heat kernel expansion]\cite[Theorem~4.1.4]{Duesberg2024_master}\label{thm:heat_kernel_rp}
Let $m \in \mathbb{N}_{\ge 2}, r > 0$. The heat kernel of $\mathbb{RP}^{m}(r)$ takes the form
$$p([x],[y],t) = \sum_{l=0,l\;\text{mod}\;2=0}^{\infty}e^{-\frac{1}{2}l(l+m-1)\frac{t}{r^2}} \frac{2}{r^{m}\mathcal{A}_{\mathbb{S}^{m}}}\frac{2l+m-1}{m-1}  C_{l}^{(m-1)/2}\left(\left\langle \frac{x}{r},\frac{y}{r} \right\rangle\right)\, ,$$
where $C_{l}^{(m-1)/2}$ denotes the Gegenbauer polynomials, see e.g. \cite{AS1972}, and $\mathcal{A}_{\mathbb{S}^{m}}$ denotes the volume of $\mathbb{S}^{m}$.
\end{Thm}
\begin{proof}
The heat kernel expansion on $\mathbb{S}^m(r)$ is
$$p([x],[y],t) = \sum_{l=0}^{\infty}e^{-\frac{1}{2}l(l+m-1)\frac{t}{r^2}} \frac{1}{r^{m}\mathcal{A}_{\mathbb{S}^{m}}}\frac{2l+m-1}{m-1}  C_{l}^{(m-1)/2}\left(\left\langle \frac{x}{r},\frac{y}{r} \right\rangle\right)\, .$$
$C_{l}^{(m-1)/2}$ is symmetric for even $l$ and antisymmetric for odd $l$, therefore only the even terms, which are invariant under the antipodal point symmetry that is factored out, contribute for the real projective space. Also, the total volume of the real projective space is only $\mathcal{A}_{\mathbb{S}^{m}} / 2$.

Plugging in these modifications, yields the result
\begin{align*}
p([x],[y],t) = \sum_{l=0, l \Mod 2 = 0}^{\infty} e^{-l(l+m-1)\frac{t}{2r^2}}\frac{2} {r^{m}\mathcal{A}_{\mathbb{S}^{m}}}\frac{2l+m-1}{m-1} C_{l}^{(m-1)/2}\left(\left\langle \frac{x}{r},\frac{y}{r} \right\rangle\right)
\end{align*}
\end{proof}

Since the diffusion mean sets can contain multiple points, one cannot simply investigate convergence of point sequences. Instead we need to consider appropriate notions of set convergence. In the present context, the notions of asymmetrical subset convergence described by the Kuratowski limits are used. Indeed, the results we show here cannot be strengthened to hold for limits based on the Hausdorff distance without significant additional restrictions.

\begin{Def}[Set convergence]\cite{kuratowski}[p.335ff] \cite{royset2020}\\ \cite{rockafellar2009}[p.109]\label{def:set_conv}\\
Let $(Q,d)$ be a metric space, $(A_{n})_{n \in \mathbb{N}} \subseteq Q$ be a sequence of sets in $Q$. Let  $N_{\infty}^{\#}$ denote the set of all infinite, strictly monotone increasing sequences in $\mathbb{N}$.
Then the \emph{Kuratowski lower limit} $\text{Li}_{n \to \infty}$ and \emph{Kuratowski upper limit} $\text{Ls}_{n\to\infty}$ of $(A_{n})_{n\in\mathbb{N}}$ are defined by
\begin{align*}
&\text{Li}_{n \to \infty} A_{n} := \{x \in Q| \exists (x_{n})_{n \in \mathbb{N}} \subseteq Q, x_{n} \in A_{n}:\lim_{n \to \infty} x_{n} = x \},\\ 
&\text{Ls}_{n \to \infty} A_{n} := \{x \in Q| \exists N \in N_{\infty}^{\#} \exists (x_{n})_{n \in N} \subseteq Q, x_{n} \in A_{n}: \lim_{n \to \infty, n\in N} x_{n} = x \}.
\end{align*}
%If $\text{LimInn} A_{n} =\text{LimOut} A_{n}$, the \textbf{limit} of $(A_{n})_{n}$ is defined by
%$$ \text{Lim} A_{n} = \text{LimInn} A_{n} =\text{LimOut} A_{n}.$$
\end{Def}

Since the limit $t \to \infty$ runs over an uncountable set, we slightly extend the definition of Kuratowski limits as limits for all divergent countable sequences $\{t_k\}_k \subset \mathbb{R}_{\ge 0}$.

\begin{Def}\cite[Definition~2.5.15]{Duesberg2024_master}\label{def:limits}
Let $(M,g)$ be a connected compact Riemannian mani\-fold, $M$ equipped with a metric $d$ (not necessarily induced by the Riemannian metric). Let $X$ be a random variable mapping to $M$. Then define the %\textbf{inner and outer long diffusion time diffusion mean set} 
\emph{lower} and \emph{upper long diffusion time diffusion mean limit} sets by 
$$ E_{\infty}^{lower}(X) := \{y \in M|\forall 0 < t_k \stackrel{k\to \infty}{\to}\infty: y \in \text{Li}_{k \to \infty} E_{t_{k}}(X) \} = \bigcap_{t_k \stackrel{k\to \infty}{\to}\infty} \text{Li}_{k \to \infty} E_{t_{k}}(X),$$
$$ E_{\infty}^{upper}(X) := \{y \in M|\forall 0 < t_k \stackrel{k\to \infty}{\to}\infty: y \in \text{Ls}_{k \to \infty} E_{t_{k}}(X) \} = \bigcap_{t_k \stackrel{k\to \infty}{\to}\infty} \text{Ls}_{k \to \infty} E_{t_{k}}(X),$$
where $0 < t_k \stackrel{k\to \infty}{\to}\infty$ denotes a monotone increasing, diverging sequence. 
\end{Def}

As a final piece of preparation, we show an important result for the Kuratowski set limits that will be extensively applied below.

\begin{Lem}\cite[Lemma~2.5.16]{Duesberg2024_master}\label{lem:equal_lims}
Let $(M,g)$ be a connected compact Riemannian mani\-fold equipped with a metric $d$ (not necessarily induced by the Riemannian metric tensor $g$). 
Let $X$ be a random variable mapping to $M$ and $f:M \to \mathbb{R}_{0}^{+}$ be continuous w.r.t. the metric $d$ on $M$. Define the sets $$M_{\delta} := \{y \in M: f(y) > \delta\},\; M_{0} := \{y \in M: f(y) = 0\}.$$
Suppose that $M_{0}$ is not empty and that the following condition holds:
$$\forall \delta > 0 \exists t_{\delta} > 0: \forall \tilde{y} \in M_{0}, \forall y\in M_{\delta}:$$
$$\mathbb{E}\left[-\ln(p(X,y,t))\right] -
\mathbb{E}\left[-\ln(p(X,\tilde{y},t))\right] > 0
\text{ for all } t \geq t_{\delta}.$$
Then $E_{\infty}^{lower}(X)$ and $E_{\infty}^{upper}(X)$ are contained in $M_{0}$.
%and $(E_{t}(X))_{t > 0}$ converges to $M_{0}$ in the sense of Ziezold.
\end{Lem}
\begin{proof}
Note that since $E_{\infty}^{lower}(X) \subseteq E_{\infty}^{upper}(X)$, it suffices to show that $E_{\infty}^{upper}(X) \subseteq M_{0}$. 
If $E_{\infty}^{upper}(X) = \emptyset$, the statement is trivially true.\\
Assume $E_{\infty}^{upper}(X) \neq \emptyset$ and let $y \in E_{\infty}^{upper}(X)$.  
Then, by definition of  $E_{\infty}^{upper}(X)$ (see Definition~\ref{def:limits}), for any sequence 
$ 0 < t_{l} \stackrel{l\to \infty}{\to} \infty$ there exists a sequence $(y_{t_{l}})_{l}, y_{t_{l}} \in E_{t_{l}}(X)$ which has a subsequence $(y_{t_{k}})_{k}$ converging to $y$ with respect to the metric $d$. \\
Fix any sequence $0 < t_{l} \stackrel{l\to \infty}{\to} \infty$ and let  $(t_{k})_{k} \subseteq (t_{l})_{l}$ denote the subsequence for which the sequence $(y_{t_{k}})_{k}, y_{t_{k}} \in E_{t_{k}}(X)$ converges to $y$. \\
Since $f$ is continuous with respect to the metric $d$, it holds that
$$\lim_{k \to \infty} f(y_{t_{k}}) = f(y).$$
Assume that $f(y) = \delta > 0$. Then there exists a $K \in \mathbb{N}$, such that for all $k \in \mathbb{N}, k \geq K$:
$$f(y_{t_{k}}) > \frac{\delta}{2},$$
which implies that $y_{t_{k}} \in M_{\frac{\delta}{2}}$ for all $k \geq K$.\\
To $\frac{\delta}{2} > 0$ exists a $t_{\frac{\delta}{2}} > 0$, such that 
$$\forall \tilde{y} \in M_{0}, \forall y\in M_{\frac{\delta}{2}}:
\mathbb{E}\left[-\ln(p(X,y,t))\right] -\mathbb{E}\left[-\ln(p(X,\tilde{y},t))\right]
 > 0
\text{ for all } t \geq t_{\frac{\delta}{2}}.$$
Choose a $y_{t_{k}}$ such that $k \geq K$ and $t_{k} > t_{\frac{\delta}{2}}$. Then it holds that
$$ \mathbb{E}\left[-\ln(p(X,y_{t_{k}},t_{k}))\right] - \mathbb{E}\left[-\ln(p(X,\tilde{y},t_{k}))\right] > 0$$
for all $\tilde{y} \in M_{0}$, of which there exists at least one, since $M_{0}$ is assumed to be non-empty.
However, since $y_{t_{k}} \in E_{t_{k}}(X)$, by definition of the diffusion mean sets $y_{t_{k}}$ is a minimizer in $M$ of the function
$$y \to \mathbb{E}\left[-\ln(p(X,y,t_{k}))\right],$$
and therefore 
$$\mathbb{E}\left[-\ln(p(X,y_{t_{k}},t_{k}))\right] -\mathbb{E}\left[-\ln(p(X,\tilde{y},t_{k}))\right]\leq 0$$
for all $\tilde{y} \in M_{0}$.\\
This yields the contradiction
$$0 \geq \mathbb{E}\left[-\ln(p(X,y_{t_{k}},t_{k}))\right] -  \mathbb{E}\left[-\ln(p(X,\tilde{y},t_{k}))\right]> 0 \text{ for all }\tilde{y} \in M_{0}.$$
Therefore $f(y) = 0$ must hold and, in consequence, $y \in M_{0}$.
\end{proof}

\section{Convergence of Diffusion Means on Real Projective Spaces}
\label{sec:realproj}

The convergence result for the long time diffusion means requires extensive technical groundwork. We first show a bound on the Gegenbauer polynomials and use it to show that a Taylor expansion for the likelihood can be expressed in terms of convergent series. Then we use this Taylor expansion to derive an expression for the long time diffusion means which we can use to prove the desired set convergence result.

\subsection{Bounds on the Gegenbauer Polynomials}

This section explicitly derives a bound on the Gegenbauer polynomials to achieve convergence results for the relevant series derived from the heat kernel expansion below. The relevant bound is presented in Lemma~\ref{lem:Gegenbauer_bound}.

\begin{Lem}\cite[Lemma~3.2.3]{Duesberg2024_master}\label{lem:gamma_bound}
For $k \in (\mathbb{N} + \{\frac{1}{2}\})$ one has
$$ \Gamma(k) = \sqrt{\pi} \prod_{i=1}^{\lfloor k \rfloor} (k-i).$$
\end{Lem}
\begin{proof}
The statement will be proven inductively.\\
\textbf{Base case} ($k= 1 + \frac{1}{2}$)\\
Since $\Gamma\left(\frac{1}{2}\right) = \sqrt{\pi}$ and $\Gamma(x+1) = x\Gamma(x)$ for all $x \in \mathbb{R}_{+}$, holds
$$\Gamma\left(1 + \frac{1}{2}\right) = \Gamma\left(\frac{1}{2}\right) \frac{1}{2}= \sqrt{\pi} \frac{1}{2} = \sqrt{\pi} \prod_{i=1}^{\lfloor 1.5 \rfloor} \left(\left(1+\frac{1}{2}\right)-i\right).$$
\textbf{Induction step}\\
Suppose the statement is true for a $(k-1) \in (\mathbb{N} + \{\frac{1}{2}\})$. Then 
\begin{align*} \Gamma(k) &= (k-1) \Gamma(k-1)= (k-1) \sqrt{\pi} \prod_{i=1}^{\lfloor k-1 \rfloor} ((k-1)-i) \\
&= \sqrt{\pi} \prod_{i=0}^{\lfloor k-1 \rfloor} ((k-1)-i)
= \sqrt{\pi} \prod_{i=0}^{\lfloor k-1 \rfloor} (k-(i+1)) \\
&= \sqrt{\pi} \prod_{i=1}^{\lfloor k-1 \rfloor+1} (k-i) =\sqrt{\pi}\prod_{i=1}^{\lfloor k \rfloor} (k-i).
\end{align*}
\end{proof}

\begin{Lem}\cite[Lemma~3.2.4]{Duesberg2024_master}\label{lem:gegenbauer_aux}
For all $k \in \big(\mathbb{N} \cup (\mathbb{N} + \{\frac{1}{2}\})\big)$:
$$ \frac{1}{4} (\lfloor k \rfloor -1)!\leq \Gamma(k) \leq 2(\lceil k \rceil -1)!.$$
\end{Lem}
\begin{proof}
\textit{Case 1}($k \in \mathbb{N}$):\\
$$ \frac{1}{4} (\lfloor k \rfloor -1)! \leq (\lfloor k \rfloor -1)! = (k-1)! = \Gamma(k) = (k-1)! = (\lceil  k \rceil - 1)! \leq  2(\lceil k \rceil - 1)!.$$\\
\textit{Case 2}($k \in \big(\mathbb{N} + \{\frac{1}{2}\})\big)$:\\
Using Lemma~\ref{lem:gamma_bound}, one sees that 
$$\Gamma(k) = \sqrt{\pi} \prod_{i=1}^{\lfloor k \rfloor} (k-i) \leq 2 \prod_{i=1}^{\lfloor k \rfloor}(\lceil k \rceil - i) = 2 (\lceil k \rceil -1)!$$
as well as
\begin{align*}
\Gamma(k) = \sqrt{\pi} \prod_{i=1}^{\lfloor k \rfloor} (k-i) 
&= \begin{cases}
\sqrt{\pi} \frac{1}{2}
\prod_{i=1}^{\lfloor k \rfloor - 1} (k-i)  & \text{ if } k > 1 + \frac{1}{2} \\  
\sqrt{\pi} \frac{1}{2} & \text{ if } k = 1 + \frac{1}{2} 
\end{cases} \\
&\geq \begin{cases}
\frac{1}{4} \prod_{i=1}^{\lfloor k \rfloor - 1} (\lfloor k \rfloor -i)  & \text{ if } k > 1 + \frac{1}{2}\\
\frac{1}{4} & \text{ if } k = 1 + \frac{1}{2} 
\end{cases}\\
&=\begin{cases}
\frac{1}{4} (\lfloor k \rfloor -1)! & \text{ if } k > 1 + \frac{1}{2} \\
\frac{1}{4}(\lfloor 1.5 \rfloor - 1)! & \text{ if } k = 1 + \frac{1}{2}. 
\end{cases}
\end{align*}
\end{proof}

\begin{Lem}\cite[Lemma~3.2.5]{Duesberg2024_master}\label{lem:Gegenbauer_bound}
For $m\in \mathbb{N}_{\ge 2}, l \in \mathbb{N}_{0}$ the Gegenbauer polynomials $C_{l}^{(m-1)/2}$ are bounded on the interval $[-1,1]$ with the following bounds:
$$
\left|C_{l}^{(m-1)/2}\right| \leq 
\begin{cases}
1 & \text{ if } l=0 \\
2^5(l+m-2)^{m} & \text{ if } l \in \mathbb{N}.
\end{cases}  
$$
\end{Lem}
\begin{proof}
By \cite{atkinson}[p.46,p.72] holds for all $x \in [-1,1], \; m \geq 2$:
$$|C_{l}^{(m-1)/2}(x)|= 
\begin{cases}
|1|  & \text{ if } l=0 \\
\left|(m-1)x\right| & \text{ if } l=1
\end{cases}
\leq
\begin{cases}
1 & \text{ if } l=0 \\
2^5(1+m-2)^m & \text{ if } l=1.
\end{cases}
$$
By Lemma 3 of \cite{zhao} for $m\geq 2, l \geq 2$:
$$ |C_{l}^{{(m-1)}/2}| \leq \max\left(\frac{\Gamma(l+m-1)}{\Gamma(m-1)\Gamma(l+1)},\frac{\Gamma(\frac{l}{2} + \frac{m-1}{2})}{\Gamma(\frac{m-1}{2})\Gamma(\frac{l}{2} +1)}\right).$$
To find a bound for the right-hand side, one can use Lemma~\ref{lem:gegenbauer_aux} and the well-known identity $\Gamma(k) = (k-1)!$ for $k \in \mathbb{N}$ to find for $m \geq 4$ and $l \geq 2$:
$$ \frac{\Gamma(l+m-1)}{\Gamma(m-1)\Gamma(l+1)} = \frac{(l+m-2)!}{(m-2)!l!} 
= \frac{(l+m-2)\cdot\cdot\cdot(l+1)}{(m-2)!}
\leq 2^5(l+m-2)^m
$$
and
$$ \frac{\Gamma(\frac{l}{2} + \frac{m-1}{2})}{\Gamma(\frac{m-1}{2})\Gamma(\frac{l}{2}+1)} 
\leq  \frac{2\lceil \frac{l+m-3}{2}\rceil!}{\frac{1}{4}\lfloor\frac{m-3}{2} \rfloor! \frac{1}{4}\lfloor\frac{l}{2}\rfloor!}
= 2^5 \frac{\lceil\frac{l+m-3}{2} \rceil\cdots(\lfloor \frac{l}{2} \rfloor + 1)}{\lfloor\frac{m-3}{2}\rfloor!}
\leq 2^5(l+m-2)^m.
$$
It remains to find bounds for the cases of $l \geq 2$ with $m=2$ or $m=3$:
$$  \frac{\Gamma(l+m-1)}{\Gamma(m-1)\Gamma(l+1)} = \frac{(l+m-2)!}{(m-2)!l!} = 
\begin{cases}
1   & \text{ if } m=2 \\
(l+1)  & \text{ if } m=3
\end{cases}
\leq 
\begin{cases}
2^5(l+2-2)^2 & \text{ if } m=2 \\
2^5(l+3-2)^3 & \text{ if } m=3,
\end{cases}
$$
and  (using Lemma~\ref{lem:gegenbauer_aux} again):
\begin{align*}
\frac{\Gamma(\frac{l}{2} + \frac{m-1}{2})}{\Gamma(\frac{m-1}{2})\Gamma(\frac{l}{2} +1)} 
= 
\begin{cases}
\frac{\Gamma(\frac{l}{2} + \frac{1}{2})}{\Gamma(\frac{1}{2})\Gamma(\frac{l}{2} +1)}& \text{ if } m=2  \\
1  & \text{ if } m=3
\end{cases}
&\leq 
\begin{cases}
\frac{1}{\sqrt{\pi}} \frac{2\lceil \frac{l}{2} -  \frac{1}{2}\rceil!}{\frac{1}{4}\lfloor \frac{l}{2} \rfloor!} & \text{ if } m=2 \\
1 & \text{ if } m=3 \\
\end{cases}\\
&\leq \begin{cases}
2^5(l+2-2)^2 & \text{ if } m=2 \\
2^5(l+3-2)^3 & \text{ if } m=3.
\end{cases}
\end{align*}
\end{proof}

\subsection{Taylor Expansion of the Likelihood}

\begin{Notation}\label{not:cmt}
To simplify notation, for any $m \in \mathbb{N}_{\ge 2}, l \in \mathbb{N}_{0}$, and $t > 0$ define the functions
\begin{align*}
    c_{m,l}^{t}: \mathbb{RP}^{m} \times \mathbb{RP}^{m} \to \mathbb{R}
\end{align*}
to be $c^{t}_{m,l}([x],[y]) := 0$ if $l$ is odd and $l = 0$, and
\begin{align*}
c^{t}_{m,l}([x],[y]) := e^{-l(l+m-1)\frac{t}{2}} \frac{2l+m-1}{m-1} C_{l}^{(m-1)/2}(\langle x,y \rangle)
\end{align*}
for even $l> 0$.
\end{Notation}

\begin{Lem}\cite[Lemma~4.2.3]{Duesberg2024_master}\label{lem:bound}
For any $m \in \mathbb{N}_{\ge 2},l \in \mathbb{N}_{0}, t > 0$ the functions
$$c_{m,l}^{t}: \mathbb{RP}^{m} \times \mathbb{RP}^{m} \to \mathbb{R}$$
from \ref{not:cmt} are bounded on $\mathbb{RP}^{m} \times \mathbb{RP}^{m}$ with the following bounds:
\[ |c_{m,l}^{t}| \leq e^{-l(l+m-1)\frac{t}{2}} 2^7 (l+m-2)^{m+1}. \]
Here, $c_{m,0}^{t} = c_{m,1}^{t} = 0$ for all $m \in \mathbb{N}_{\ge 2}$ and for all $t > 0$. \\
Moreover, for any $m \in \mathbb{N}_{\ge 2}$, any $t > 0$ and any $[x],[y] \in \mathbb{RP}^{m}$ the following equality holds:
$$\left(\sum_{l=0}^{\infty}c^{t}_{m,l}([x],[y])\right)^2 = \sum_{n=0}^{\infty}d_{m,n}^{t}([x],[y]),$$
where $d_{m,n}^{t}: \mathbb{RP}^{m} \times \mathbb{RP}^{m} \to \mathbb{R}$ is defined for any $m \in \mathbb{N}_{\ge 2}, n\in \mathbb{N}_{0}$ and $t > 0$ by
$$d_{m,n}^{t}([x],[y]) := \sum_{l=0}^{n} c_{m,l}^{t}([x],[y])c_{m,n-l}^{t}([x],[y]).$$ 
For any $m\in\mathbb{N}_{\ge 2},n \in \mathbb{N}_{0}$ and $t > 0$ the function $d_{m,n}^{t}$ is bounded on $\mathbb{RP}^{m} \times \mathbb{RP}^{m}$ with the following bounds:
$$ |d_{m,n}^{t}| \leq e^{-(\frac{n^2}{2}+n(m-1))\frac{t}{2}}2^{14} (n+m-2)^{2m+3} .$$
Here, $d_{m,0}^{t} = d_{m,1}^{t} = d_{m,2}^{t} = d_{m,3}^{t} = 0$ for all $m \in \mathbb{N}_{\ge 2}$ and for all $t > 0$.\\
Further, for any $m\in\mathbb{N}_{\ge 2},t > 0$ and $[x],[y] \in \mathbb{RP}^{m}$ the following equality holds:
$$\left(\sum_{l=0}^{\infty}c^{t}_{m,l}([x],[y])\right)^3 = \sum_{k=0}^{\infty} h_{m,k}^{t}([x],[y]),$$
where $h_{m,k}^{t}: \mathbb{RP}^{m} \times \mathbb{RP}^{m} \to \mathbb{R}$ is defined for any $m \in \mathbb{N}_{\ge 2},k \in \mathbb{N}_{0}$ and $t > 0$ by
$$h_{m,k}^{t}([x],[y]) := \sum_{n=0}^{k} d_{m,n}^{t}([x],[y])c_{m,k-n}^{t}([x],[y]).$$
For any $m \in \mathbb{N}_{\ge 2},k \in \mathbb{N}_{0}$ and $t > 0$ the function $h_{m,k}^{t}$ is bounded on $\mathbb{RP}^{m} \times \mathbb{RP}^{m}$ with the following bounds: 
$$ |h_{m,k}^{t}| \leq e^{-(\frac{k^2}{3} + k(m-1))\frac{t}{2}}2^{21}(k+m-2)^{3m+5} .$$
Here, $h_{m,0}^{t}=h_{m,1}^{t}=h_{m,2}^{t}=h_{m,3}^{t}=h_{m,4}^{t}=h_{m,5}^{t}=0$ for all $m \in \mathbb{N}_{\ge 2}$ and for all $t > 0$.
\end{Lem}
\begin{proof}
Using
$$p([x],[y],t) = \frac{2}{\mathcal{A}_{\mathbb{S}^m}}\left(1+\sum_{l=1}^{\infty}c_{m,l}^t([x],[y])\right),$$
one sees that the sum $\sum_{l=0}^{\infty} c_{m,l}^{t}([x],[y])$ converges absolutely for every choice of $[x],[y] \in \mathbb{RP}^{m}, t > 0$ since the heat kernel expansion does, see Theorem~\ref{thm:heat_expansion}. Therefore, the Cauchy product $\sum_{n=0}^{\infty}d_{m,n}^{t}([x],[y])$ of this sum with itself converges absolutely (see \cite{tretter}[p.63]) for any choice of $[x],[y] \in \mathbb{RP}^{m},t > 0$ to $(\sum_{l=0}^{\infty} c_{m,l}^{t}([x],[y]))^2$. The Cauchy product $\sum_{k=0}^{\infty}h_{m,k}^{t}([x],[y])$ of the absolutely convergent sums $\sum_{n=0}^{\infty}d_{m,n}^{t}([x],[y])$ and $\sum_{l=0}^{\infty} c_{m,l}^{t}([x],[y])$ converges (absolutely) to 
$$\left(\sum_{l=0}^{\infty}c^{t}_{m,l}([x],[y])\right)^2\left(\sum_{l=0}^{\infty}c^{t}_{m,l}([x],[y])\right) = \left(\sum_{l=0}^{\infty}c^{t}_{m,l}([x],[y])\right)^3.$$
The first inequality is a consequence of Lemma~\ref{lem:Gegenbauer_bound} bounding the Gegenbauer polynomials $C_{l}^{(m-1)/2}$.
\begin{align*}
\max_{x,y \in \mathbb{S}^{m}}|c_{m,l}^{t}(x,y)| &= \max_{x,y \in \mathbb{S}^{m}}\left|e^{-l(l+m-1)\frac{t}{2}} \frac{2l+m-1}{m-1}C_{l}^{(m-1)/2}(\langle x,y \rangle)\right| \\
&\leq e^{-l(l+m-1)\frac{t}{2}} \frac{2l+m-1}{m-1}2^5(l+m-2)^{m} \\
&\leq  e^{-l(l+m-1)\frac{t}{2}} 4(l+m-2) 2^5(l+m-2)^{m} \\
&\leq e^{-l(l+m-1)\frac{t}{2}}2^7 (l+m-2)^{m+1}.
\end{align*}

The bounds for $d_{m,n}^{t}$ and $h_{m,k}^{t}$ are derived for $k,n \in \mathbb{N} \setminus \{1\}$ from the bound of $c_{m,l}^{t}$ as follows (noting that always $c_{m,l}^{t} = 0$):\\
\begin{align*}
|d_{m,n}^{t}(x,y)| &= \left|\sum_{l=0}^{n} c_{m,l}^{t}(x,y)c_{m,n-l}^{t}(x,y)\right| \leq \sum_{l=1}^{n-1}|c_{m,l}^{t}(x,y)||c_{m,n-l}^{t}(x,y)|\\
&\leq n2^{14}\max_{l\in \{1,...,n-1\}}(l+m-2)^{2m+2}\max_{l\in \{1,...,n-1\}}(e^{-l(l+m-1)\frac{t}{2}}e^{-(n-l)((n-l)+m-1)\frac{t}{2}}) \\
&\leq 2^{14}(n+m-2)^{2m+2}n\max_{l\in \{1,...,n-1\}}e^{-(l^2+(n-l)^2 + n(m-1))\frac{t}{2}}\\
&\leq 2^{14}(n+m-2)^{2m+3}e^{-(\frac{n^2}{2}+n(m-1))\frac{t}{2}},
\end{align*}
where it was used again that $l^2 +(n-l)^2$ is minimized in $[0,n]$ w.r.t. $l$ by $\frac{n}{2}$,
and 
\begin{align*}
|h_{m,k}^{t}(x,y)| &= \left|\sum_{n=0}^{k} d_{m,n}^{t}(x,y)c_{m,k-n}^{t}(x,y)\right| 
\leq \sum_{n=0}^{k-1}|d_{m,n}^{t}(x,y)||c_{m,k-n}^{t}(x,y)|\\
&\leq k 2^{21} \max_{n\in \{0,...,k-1\}}(n+m-2)^{3m+4} \max_{n\in \{0,...,k-1\}} (e^{-(\frac{n^2}{2}+n(m-1))\frac{t}{2}}e^{-(k-n)((k-n)+m-1)\frac{t}{2}}) \\
&\leq 2^{21}(k+m-2)^{3m+4}k \max_{n\in \{0,...,k-1\}} e^{-(\frac{n^2}{2}+(k-n)^2 + k(m-1))\frac{t}{2}}\\
&\leq 2^{21}(k+m-2)^{3m+5} e^{-(\frac{k^2}{3}+ k(m-1))\frac{t}{2}},
\end{align*}
where it was used again that $\frac{n^2}{2} + (k-n)^2$ is minimized in $[0,k]$ w.r.t. $n$ by $\frac{2}{3}k$.\\

Further, noting that for all $m\in\mathbb{N}_{\ge 2}$ and $t > 0$ by definition holds $c_{m,0}^{t} = c_{m,1}^{t} = 0$, it holds for all $m \in \mathbb{N}_{\ge 2}$ and all $t > 0$ that:

\begin{align*}
d_{m,0}^{t} &= c_{m,0}^{t}c_{m,0}^{t} = 0 \\
d_{m,1}^{t} &= c_{m,0}^{t}c_{m,1}^{t} + c_{m,1}^{t}c_{m,0}^{t} = 0 \\
d_{m,2}^{t} &= c_{m,0}^{t}c_{m,2}^{t} + c_{m,1}^{t}c_{m,1}^{t} +
c_{m,2}^{t}c_{m,0}^{t}  = 0 \\
d_{m,3}^{t} &= 
c_{m,0}^{t}c_{m,3}^{t} + c_{m,1}^{t}c_{m,2}^{t} +
c_{m,2}^{t}c_{m,1}^{t} +
c_{m,3}^{t}c_{m,0}^{t}  = 0 \\
\end{align*}
and 
\begin{align*}
h_{m,0}^{t} &= d_{m,0}^{t}c_{m,0}^{t} = 0 \\
h_{m,1}^{t} &= d_{m,0}^{t}c_{m,1}^{t}+
d_{m,1}^{t}c_{m,0}^{t} = 0 \\
h_{m,2}^{t} &= d_{m,0}^{t}c_{m,2}^{t}+
d_{m,1}^{t}c_{m,1}^{t}+
d_{m,2}^{t}c_{m,0}^{t} = 0 \\
h_{m,3}^{t} &= d_{m,0}^{t}c_{m,3}^{t}+
d_{m,1}^{t}c_{m,2}^{t}+
d_{m,2}^{t}c_{m,1}^{t}+
d_{m,3}^{t}c_{m,0}^{t} = 0 \\
h_{m,4}^{t} &= d_{m,0}^{t}c_{m,4}^{t}+
d_{m,1}^{t}c_{m,3}^{t}+
d_{m,2}^{t}c_{m,2}^{t}+
d_{m,3}^{t}c_{m,1}^{t}+
d_{m,4}^{t}c_{m,0}^{t} = 0\\
h_{m,5}^{t} &= d_{m,0}^{t}c_{m,5}^{t}+
d_{m,1}^{t}c_{m,4}^{t}+
d_{m,2}^{t}c_{m,3}^{t}+
d_{m,3}^{t}c_{m,2}^{t}+
d_{m,4}^{t}c_{m,1}^{t}+
d_{m,5}^{t}c_{m,0}^{t} = 0.
\end{align*}
\end{proof}

\begin{Thm}\cite[Theorem~4.2.4]{Duesberg2024_master}\label{thm:taylor_ln}
Let $m \in \mathbb{N}_{\ge 2}$. Then there exists a function $g:(-1,\infty) \to \mathbb{R}$ and $t_{0} > 0, C > 0$ such that for all $t > t_{0}$ and for all $[x],[y] \in \mathbb{RP}^{m}$:
$$\ln\left(1+\sum_{l=1}^{\infty}c_{m,l}^{t}([x],[y])\right) = \sum_{l=2}^{\infty} c_{m,l}^{t}([x],[y])-\frac{1}{2}\sum_{n=4}^{\infty}d_{m,n}^{t}([x],[y]) + g\left(\sum_{l=1}^{\infty}c_{m,l}^{t}([x],[y])\right),$$
where $|g\left(\sum_{l=1}^{\infty}c_{m,l}^{t}([x],[y])\right)| < C |\sum_{k=6}^{\infty} h_{m,k}^{t}([x],[y])|$.
\end{Thm}
\begin{proof}
The expression
$$\max_{[x],[y] \in \mathbb{RP}^{m}}\left|\sum_{l=1}^{\infty} c_{m,l}^{t}([x],[y])\right| $$
can be bounded using the bounds of $c_{m,l}^{t}$ from Lemma~\ref{lem:bound} by
$$
\max_{[x],[y] \in \mathbb{RP}^{m}}\left|\sum_{l=1}^{\infty} c_{m,l}^{t}([x],[y])\right| \leq
\sum_{l=1}^{\infty}e^{-l(l+m-1)\frac{t}{2}} 2^7 (l+m-2)^{m+1} .
$$
Using the quotient criterion for infinite sums (see \cite{tretter}[p.63]), one sees that there exist a $T > 0$ such that $\sum_{l=1}^{\infty}e^{-l^2\frac{t}{2}}2$ and $\sum_{l=1}^{\infty}e^{-l(l+m-1)\frac{t}{2}} 2^7 (l+m-2)^{m+1}$ converge for all fixed $t > T$.
Since all of their summands tend to $0$ for $t \to \infty$, so do the sums. In consequence:
$$ \max_{[x],[y] \in \mathbb{RP}^{m}}\left|\sum_{l=1}^{\infty} c_{m,l}^{t}([x],[y])\right|  
\stackrel{t \to \infty}{\longrightarrow} 0.$$
With the Taylor expansion of $\ln(1+x)$ at $x=0$ (see \cite{tretter}[p.139]) given by
$$\ln(1+x) = x-\frac{1}{2}x^2 + O(x^3),$$
for the function $$g:(-1,\infty) \to \mathbb{R},\;g(x) := \ln(1+x)-\left(x-\frac{1}{2}x^2\right)$$ there exist 
$\varepsilon > 0,C_{\varepsilon} > 0$ such that $|g(x)| \leq C_{\varepsilon} |x|^3$ for all $x \in (-1,\infty)$ with $|x| \leq \varepsilon$.
Choosing 
$t_{0}$ such that for all $t > t_{0}$ 
$$ \max_{[x],[y] \in \mathbb{RP}^{m}}\left|\sum_{l=1}^{\infty} c_{m,l}^{t}([x],[y])\right| \leq \varepsilon,$$
the Taylor expansion can be applied to yield for all $[x],[y] \in \mathbb{RP}^{m}$ and $t > t_{0}$:
$$\ln\left(1+\sum_{l=1}^{\infty} c^{t}_{m,l}([x],[y])\right) =\sum_{l=1}^{\infty} c^{t}_{m,l}([x],[y])-\frac{1}{2}\left(\sum_{l=1}^{\infty} c^{t}_{m,l}([x],[y])\right)^2 + g\left( \sum_{l=1}^{\infty} c^{t}_{m,l}([x],[y]) \right)$$
where $\left| g\left(\sum_{l=1}^{\infty} c^{t}_{m,l}([x],[y])\right)\right| < C\left|\sum_{l=1}^{\infty}c^{t}_{m,l}([x],[y])\right|^3$.\\
Plugging in the expressions for $\left(\sum_{l=0}^{\infty} c^{t}_{m,l}([x],[y])\right)^2$ and $\left(\sum_{l=0}^{\infty} c^{t}_{m,l}([x],[y])\right)^3$ from Lemma~\ref{lem:bound}
and noting that for all $t > 0$ it holds that
\begin{align*}
c_{m,0}^{t} = c_{m,1}^{t} 
= d_{m,0}^{t} = d_{m,1}^{t} = d_{m,2}^{t} = d_{m,3}^{t}  
= h_{m,0}^t = h_{m,1}^{t} = h_{m,2}^{t} = h_{m,3}^{t} = h_{m,4}^{t} = h_{m,5}^{t} = 0
\end{align*}
yields the desired equation.
\end{proof}

The usefulness of this expansion is not immediately obvious. Simply stated, it serves to separate powers of $e^{-t}$. Below, we use the fact that in the limit $t \to \infty$, only the smallest power of $e^{-t}$ contributes. From the above one can see that
\begin{align*}
    c_{m,2}^{t}([x],[y]) =&~ e^{-(m+1)t} \frac{m+3}{m-1} C_{2}^{(m-1)/2}(\langle x,y \rangle)\\
    =&~ e^{-(m+1)t} \frac{m+3}{2} \left((m+1) \langle x,y \rangle^2 - 1\right)
\end{align*}
is the lowest order term, which will therefore determine the long time limit of the diffusion means.

\subsection{Convergence of Diffusion Mean Sets}

In the following, we will prove several convergence results, so it is necessary to introduce the metric, in which these convergence results are expressed. The metrics discussed here, as well as the associated intrinsic, extrinsic and residual means have been discussed by \cite{H_meansmeans_12}.

\begin{Def}[Metrics on Spheres and Real Projective Spaces]
Let $m \in \mathbb{N}_{m\ge2 }, r > 0$. The spheres $\mathbb{S}^m(r)$ can be equipped with the
\begin{enumerate}
    \item \emph{chordal metric}, defined as $d_{\text{chord}}(x,y) := |x-y|$ in the embedding space,
    \item \emph{geodesic metric} $d_{\text{geo}}(x,y) = r \arccos(r^{-2}\langle x,y \rangle)$, defined as the length of the shortest geodesic segment connecting two points.
\end{enumerate}
The real projective space $\mathbb{RP}^{m}(r)$ can be equipped with the
\begin{enumerate}\setcounter{enumi}{2}
    \item \emph{geodesic metric} $d_{\text{geo-p}} ([x],[y]) := \min(d_{\text{geo}}(x,y), d_{\text{geo}}(x,-y))$ inherited from the sphere,
    \item \emph{residual metric} $ d_{\text{res}}([x],[y]) := \sqrt{1-\left(\left(\frac{x}{r}\right)^{T}\left(\frac{y}{r}\right)\right)^2}$.
\end{enumerate}
\end{Def}

Note that $d_{\text{chord}}$ and $d_{\text{geo}}$ are equivalent on the sphere and $d_{\text{geo-p}}$ and $d_{\text{res}}$ are equivalent on the real projective space. Furthermore, due to the connection between $d_{\text{geo}}$ and $d_{\text{geo-p}}$, any sequence $([y]_{k})_{k}$ in $\mathbb{RP}^{m}(r)$ which converges in $d_{\text{res}}$ to some $[y]$ gives rise to a convergent sequence of representatives $(y_{k} \in [y]_{k})_{k}$ in $\mathbb{S}^{m}(r)$ converging in $d_{\text{chord}}$ and $d_{\text{geo}}$ to a representative $y \in [y]$.

\begin{Thm}\cite[Theorem~4.2.5]{Duesberg2024_master}
\label{thm:diff_mean}
Let $m\in \mathbb{N}_{\ge 2}$ and let $X$ be a random variable mapping to $(\mathbb{RP}^{m},d)$, where $d$ is the residual metric. 
Then $E_{\infty}^{lower}(X)$ and $E_{\infty}^{upper}(X)$ are contained in 
$$M_{0} := \argmax_{[y] \in \mathbb{RP}^{m}}y^{T} \mathbb{E}[XX^{T}]y.$$
%and $(E_{t}(X))_{t > 0}$ converges to $M_{0}$ in the sense of Ziezold.\\
Let $D$ denote the intersection of $\mathbb{S}^{m}$ with the eigenspace of the largest  eigenvalue of $\mathbb{E}[XX^{T}]$. Then 
$$M_{0} = D/\sim $$
where
$$ x \sim y :\Leftrightarrow x = y \lor x = -y.$$
If the largest eigenvalue of $\mathbb{E}[XX^{T}]$ has multiplicity $1$, it holds that:
$$M_{0} = E_{\infty}^{lower}(X) = E_{\infty}^{upper}(X).$$
\end{Thm}
\begin{proof}
The function defined by 
$$ y \to y^T\mathbb{E}[XX^T]y$$
is continuous in $\mathbb{R}^{m+1}$ and, since the set $\mathbb{S}^{m} \subset \mathbb{R}^{m+1}$ is compact in $\mathbb{R}^{m+1}$, it admits a maximum on $\mathbb{S}^{m}$. This implies that the expression
$$\max_{z \in \mathbb{S}^{m}} z^T\mathbb{E}[XX^T]z$$
is well-defined, and since the function $y \to y^T\mathbb{E}[XX^T]y$ is even on $\mathbb{S}^{m}$, the expression
$$ \max_{[z] \in \mathbb{RP}^{m}} z^T\mathbb{E}[XX^T]z.$$
is well-defined too. \\
The set $M_{0} = \argmax_{[y] \in \mathbb{RP}^{m}}y^{T} \mathbb{E}[XX^{T}]y $ can equivalently be defined as the set of zero points of a function $f$ where 
$$f:\mathbb{RP}^{m} \to \mathbb{R}_{0}^{+}, 
f([y]) := \max_{[z] \in \mathbb{RP}^{m}} z^{T} \mathbb{E}[XX^{T}]z - y^{T}\mathbb{E}[XX^{T}]y$$
and 
$$M_{0} = \{[y] \in \mathbb{RP}^{m}:f([y]) = 0\}.$$
Define further for any $\delta > 0$ the set $M_{\delta}$ as follows:
$$M_{\delta} := \{[y] \in \mathbb{RP}^{m}:f([y]) > \delta\}.$$
In this setting, to prove that $E^{lower}_{\infty}(X)$ and $E^{upper}_{\infty}(X)$ are contained in $M_{0}$, it suffices to show that the conditions of Lemma~\ref{lem:equal_lims} are satisified for the function $f$. So it needs to be shown that the function $f$ is continuous with respect to  the residual metric on $\mathbb{RP}^{m}$, that the set $M_{0}$ is not empty and that the following condition holds:

$$ \forall \delta > 0 \exists t_{\delta} > 0: \forall [\tilde{y}] \in M_{0}, [y] \in M_{\delta}:$$
$$ \mathbb{E}\left[-\ln\left(p(X,[y],t)\right)\right] - \mathbb{E}\left[-\ln\left(p(X,[\tilde{y}],t)\right)\right] > 0 \text{ for all } t \geq t_{\delta}.$$

The function $f$ is continuous with respect to the residual metric by the following argument:\\
Let $([y]_{k})_{k}\subseteq  \mathbb{RP}^{m}$ be a sequence converging to $[y] \in \mathbb{RP}^{m}$ with respect to the residual metric. 
%By Lemma~\ref{lem:res_conv_geo_conv}, 
This implies the existence of a sequence of representatives $y_{k}$ of $[y]_{k}$ and a representative $y$ of $[y]$, such that the sequence $(y_{k})_{k}$ converges to $y$ with respect to the geodesic metric on $\mathbb{S}^{m}$. This in turn implies that $(y_{k})_{k}$ converges to $y$ with respect to the chordal metric on $\mathbb{S}^{m}$ too.% by Lemma~\ref{lem:geochor}.
Therefore it holds that:
\begin{align*} \lim_{k \to \infty} f([y]_{k}) &=  \max_{[z] \in \mathbb{RP}^{m}} z^{T} \mathbb{E}[XX^{T}] z - \lim_{k \to \infty}(y_{k})^{T}\mathbb{E}[XX^{T}]y_{k} \\
&= \max_{[z] \in \mathbb{RP}^{m}} z^{T} \mathbb{E}[XX^{T}]z - y^{T}\mathbb{E}[XX^{T}]y \\
&= f([y]),
\end{align*}
where it was used for the second equality that the function $z \to z^T\mathbb{E}[XX^{T}]z$ is continuous in $\mathbb{S}^{m} \subset \mathbb{RP}^{m+1}$ with respect to the Euclidean (=chordal) metric.\\
To see that $M_{0}$ is not empty, note that 
$$\argmax_{y \in \mathbb{S}^{m}} y^T\mathbb{E}[XX^T]y$$
is not empty, because the continuous function $y \to y^T\mathbb{E}[XX^T]y$ attains a maximum on the compact set $\mathbb{S}^{m}$. 
Because of the evenness of the function $y \to y^T\mathbb{E}[XX^T]y$, if $y$ is a maximizer of it in $\mathbb{S}^m$, so is $-y$. This shows that
$$\argmax_{[y] \in \mathbb{RP}^{m}} y^T\mathbb{E}[XX^T]y
= \{[y] \in \mathbb{RP}^{m}| y \in \argmax_{y \in \mathbb{S}^m} y^T\mathbb{E}[XX^T]y\} $$
is not empty.\\

Having shown that $f$ is continuous with respect to the residual metric on $\mathbb{RP}^{m}$ and that $M_{0}$ is not empty, to apply Lemma~\ref{lem:equal_lims} it only remains to be checked that the following condition holds:
$$ \forall \delta > 0 \exists t_{\delta} > 0: \forall [\tilde{y}] \in M_{0}, [y] \in M_{\delta}:$$
$$ \mathbb{E}\left[-\ln\left(1 + \sum_{l=1}^{\infty} c_{m,l}^{t}(X,[y])\right)\right] - \mathbb{E}\left[-\ln\left(1 + \sum_{l=1}^{\infty} c_{m,l}^{t}(X,[\tilde{y}])\right)\right] > 0 \text{ for all } t \geq t_{\delta},$$
where it was used that
$$p([x],[y],t) = \frac{2}{\mathcal{A}_{\mathbb{S}^{m}}} (1 + \sum_{l=1}^{\infty}c_{m,l}^{t}([x],[y]))$$
with $c_{m,l}^{t}$ as defined in \ref{not:cmt}.

Plugging in the Taylor expansion from Theorem~\ref{thm:taylor_ln} for the expression $\ln(1+\sum_{l=1}^{\infty}c_{m,l}^{t}([x],[y]))$, this condition is equivalent to the following condition:
$$\forall \delta > 0 \exists t_{\delta} > 0: \forall [\tilde{y}] \in M_{0}, [y] \in M_{\delta}:$$
\begin{align*}
&\mathbb{E}\left[ \sum_{l=2}^{\infty} (c^{t}_{m,l}(X,[\tilde{y}]) -c^{t}_{m,l}(X,[y]))\right] - 
\frac{1}{2}\mathbb{E}\left[ \sum_{n=4}^{\infty}(d_{m,n}^{t}(X,[\tilde{y}]) - d_{m,n}^{t}(X,[y]))\right]\\
&+\mathbb{E}\left[g\left(\sum_{l=1}^{\infty} c^{t}_{m,l}(X,[\tilde{y}])\right) -g\left(\sum_{l=1}^{\infty} c^{t}_{m,l}(X,[y])\right)\right] > 0 \text{ for all } t \geq t_{\delta}.
\end{align*}
where $g:(-1,\infty) \to \mathbb{R},\; g(z) := \ln(1+z)-\left(z-\frac{1}{2}z^2\right)$.\\
Using the fact that (see \cite{atkinson}[p.46,p.72]) for $m \in \mathbb{N}_{\ge 2}, z \in [-1,1]$:
\begin{align*}
C_{0}^{(m-1)/2}(z) &= 1 \\
C_{1}^{(m-1)/2}(z) &= (m-1)z \\
C_{2}^{(m-1)/2}(z) &= \frac{m+1}{2}zC_{1}^{(m-1)/2}(z) - \frac{m-1}{2}C_{0}^{(m-1)/2}(z) \\
&=  \frac{(m+1)(m-1)}{2}z^2 - \frac{m-1}{2}
\end{align*}
and that for $z \in [-1,1]$
\begin{align*}
\cos(2\arccos(z)) = 2\cos^2(z) -1,
\end{align*}
for any $\delta > 0, [\tilde{y}] \in M_{0}, [y] \in M_{\delta}$:
\begin{align*}
&\mathbb{E}[c_{m,2}^{t}(X,[\tilde{y}])-c_{m,2}^t(X,[y])] \\
&= \mathbb{E}[e^{-(m+1)t}\frac{m+3}{m-1}\big(C_{2}^{(m-1)/2}(\langle X,\tilde{y} \rangle) - C_{2}^{(m-1)/2}( \langle X,y \rangle) \big)]\\
&=\frac{(m+3)(m+1)}{2}e^{-(m+1)t}(\mathbb{E}[\langle X,\tilde{y}\rangle^2] -\mathbb{E}[\langle X,y\rangle^2])  \\
&= \frac{(m+3)(m+1)}{2}e^{-(m+1)t} (\tilde{y}^T\mathbb{E}[XX^T]\tilde{y} - y^T\mathbb{E}[XX^T]y)\\
&=  \frac{(m+3)(m+1)}{2}e^{-(m+1)t}f([y]),
\end{align*}
where it was used that for $[\tilde{y}] \in M_{0}$ it holds that
$$ \tilde{y}^T\mathbb{E}[XX^T]\tilde{y} = \max_{[z] \in \mathbb{RP}^{m}}z^T\mathbb{E}[XX^T]z.$$
So the condition is equivalent to the following condition:
$$ \forall \delta > 0 \exists t_{\delta} > 0: \forall [\tilde{y}] \in M_{0}, [y] \in M_{\delta}: $$
\begin{align*} 
&\mathbb{E}\left[ \sum_{l=3}^{\infty} (c^{t}_{m,l}(X,[y]) -c^{t}_{m,l}(X,[\tilde{y}]))\right] +
\frac{1}{2}\mathbb{E}\left[ \sum_{n=4}^{\infty}d_{m,n}^{t}(X,[\tilde{y}]) -\sum_{n=4}^{\infty}d_{m,n}^{t}(X,[y])\right] \\
&+\mathbb{E}\left[ g\left(\sum_{l=1}^{\infty} c^{t}_{m,l}(X,[y])\right) -g\left(\sum_{l=1}^{\infty} c^{t}_{m,l}(X,[\tilde{y}])\right)\right] \\
&< \frac{(m+3)(m+1)}{2}e^{-(m+1)t}f([y])\;\;
\text{ for all } t \geq t_{\delta}.
\end{align*}
By \ref{thm:taylor_ln} there exist $t_{0} >0, C > 0$, such that for all $t > t_{0}$
$$\left|g\left(\sum_{l=1}^{\infty}c_{m,l}^{t}([x],[y])\right)\right| < C \left|\sum_{k=6}^{\infty} h_{m,k}^{t}([x],[y])\right| \text{ for all } [x],[y] \in \mathbb{RP}^{m}.$$
Using the bounds from \ref{lem:bound} for the expressions $c_{m,l}^{t}, d_{m,l}^{t}$ and $h_{m,l}^{t}$, for all $t > t_{0}$
\begin{align*}
&\mathbb{E}\left[ \sum_{l=3}^{\infty} (c^{t}_{m,l}(X,[y]) -c^{t}_{m,l}(X,[\tilde{y}]))\right] +
\frac{1}{2}\mathbb{E}\left[ \sum_{n=4}^{\infty}d_{m,n}^{t}(X,[\tilde{y}]) -\sum_{n=4}^{\infty}d_{m,n}^{t}(X,[y])\right]\\
&+\mathbb{E}\left[ g\left(\sum_{l=1}^{\infty} c^{t}_{m,l}(X,[y])\right) -g\left(\sum_{l=1}^{\infty} c^{t}_{m,l}(X,[\tilde{y}])\right)\right]
\end{align*}
can be bounded by 
\begin{align*}
&2 \sum_{l=3}^{\infty} e^{-l(l+m-1)\frac{t}{2}} 2^7 (l+m-2)^{m+1} + \sum_{n=4}^{\infty} e^{-(\frac{n^2}{2}+n(m-1))\frac{t}{2}}2^{14} (n+m-2)^{2m+3}\\
&+ 2 C \sum_{k=6}^{\infty}e^{-(\frac{k^2}{3} + k(m-1))\frac{t}{2}}2^{21}(k+m-2)^{3m+5} .
\end{align*}
In consequence, the condition is satisfied if the following, stronger condition holds:
$$ \forall \delta > 0 \exists t_{\delta}: \forall [y] \in M_{\delta} \forall t \geq t_{\delta}: $$
\begin{align*}
&2 \sum_{l=3}^{\infty} e^{-l(l+m-1)\frac{t}{2}} 2^7 (l+m-2)^{m+1} + \sum_{n=4}^{\infty} e^{-(\frac{n^2}{2}+n(m-1))\frac{t}{2}}2^{14} (n+m-2)^{2m+3}\\
&+ 2 C\sum_{k=6}^{\infty}e^{-(\frac{k^2}{3} + k(m-1))\frac{t}{2}}2^{21}(k+m-2)^{3m+5}
<  e^{-(m+1)t}\frac{(m+3)(m+1)}{2}f([y]).
\end{align*}
Applying the quotient criterion for infinite sums (see \cite{tretter}[p.59]) one sees that there exists a $t_{conv} > t_{0}$, such that for all choices of $t \geq t_{conv}$ all the involved sums converge.\\
In consequence, for any fixed $t \geq t_{conv}$ the inequalities can be multiplied by the factor $e^{(m+1)t}$. Multiplying the inequalities with this factor and noting that for all $[y] \in M_{\delta}$ holds $f([y]) > \delta$, one sees that the previous conditions hold if the following, stronger condition holds:
$$\forall \delta > 0 \exists t_{\delta} > 0: \forall t \geq \max\{t_{\delta},t_{conv}\}:$$
\begin{align*}
\frac{2}{(m+3)(m+1)}&\Big(2 \sum_{l=3}^{\infty} e^{-(l(l+m-1)-2(m+1))\frac{t}{2}} 2^7 (l+m-2)^{m+1} \\
&+\sum_{n=4}^{\infty} e^{-(\frac{n^2}{2}+n(m-1)-2(m+1))\frac{t}{2}}2^{14} (n+m-2)^{2m+3} \\
&+ 2 C\sum_{k=6}^{\infty}e^{-(\frac{k^2}{3} + k(m-1)-2(m+1))\frac{t}{2}}2^{21}(k+m-2)^{3m+5}\Big) < \delta.
\end{align*}
Since for $t \to \infty$ all sums involved converge to $0$ (to see this, it suffices to verify that the factors of $\frac{t}{2}$ in the exponent of $e$ are negative for all indices), the condition is satisfied in both cases.
Therefore, Lemma~\ref{lem:equal_lims} can be applied to yield that $E^{lower}_{\infty}(X)$ and $E^{upper}_{\infty}(X)$ are contained in $M_{0}$.\\

For the second part of the Theorem, note that:
$$y \in \argmax_{y \in \mathbb{S}^{m}}y^{T} \mathbb{E}[XX^{T}]y$$
$$\Longleftrightarrow$$
$$y\in\mathbb{S}^{m} \text{ is an eigenvector to the largest eigenvalue of } \mathbb{E}[XX^{T}]$$
$$\Longleftrightarrow$$
$$-y\in\mathbb{S}^{m} \text{ is an eigenvector to the largest eigenvalue of } \mathbb{E}[XX^{T}].$$
Denote the eigenspace of the largest eigenvalue of $\mathbb{E}[XX^T]$  intersected with $\mathbb{S}^{m}$ by $D$. Then:
\begin{align*}
D/\sim 
&= \{[y] \in \mathbb{RP}^{m}| y \in D\} \\
&= \{[y] \in \mathbb{RP}^{m}| y \in \argmax_{y \in \mathbb{S}^{m}}y^{T} \mathbb{E}[XX^{T}]y \} \\
&= \argmax_{[y] \in \mathbb{RP}^{m}} y^{T}\mathbb{E}[XX^{T}]y = M_{0}.
\end{align*}

Let $\lambda$ denote the largest eigenvalue of $\mathbb{E}[XX^{T}]$. Since $\mathbb{E}[XX^{T}]$ is symmetric and positive semi-definite, $\lambda$ is a non-negative real number, and its algebraic and geometric multiplicity coincide. If the geometric multiplicity of $\lambda$ is $1$, the eigenspace to $\lambda$ is a one-dimensional subset in $\mathbb{R}^{m+1}$, therefore the set $D$ can only contain two elements $y_{\infty},-y_{\infty} \in \mathbb{S}^{m}$. In consequence, $$M_{0} = D / \sim = \{[y_{\infty}]\}.$$
Assume now that $\lambda$ has multiplicity $1$.
It was already established that $E^{lower}_{\infty}(X) \subseteq E^{upper}_{\infty}(X) \subseteq M_{0}$.
So to show equality of $E^{lower}_{\infty}(X)$, $E^{upper}_{\infty}(X)$ and $M_{0}$, it suffices to show that the sole element $[y_{\infty}]$ in $M_{0}$ is also contained in $E^{lower}(X)$.\\
This in turn is the case, if for any monotone increasing sequences $0 < t_{k} \stackrel{k \to \infty}{\to} \infty$ and corresponding sequences $([y]_{t_{k}})_{k}, [y]_{t_{k}} \in E_{t_{k}}(X)$ the limit $\lim_{k \to \infty} [y]_{t_{k}}$ is given by $[y_{\infty}]$. This will be shown in the following:\\

Fix a sequence $0 < t_{k} \stackrel{k \to \infty}{\to} \infty$ and build  $([y]_{t_{k}})_{k}$ by arbitrarily choosing an element from each set $E_{t_{k}}(X)$. Consider the sequence $f([y]_{t_{k}})_{k}$. It has been shown already that the conditions of Lemma~\ref{lem:equal_lims} hold for $f$. This implies that
$$ \forall \delta > 0 \exists t_{\delta} > 0: \forall [\tilde{y}] \in M_{0} \forall [y] \in M_{\delta}:$$
$$  \mathbb{E}\left[-\ln\left(p(X,[y],t\right)\right]- \mathbb{E}\left[-\ln\left(p(X,[\tilde{y}],t)\right)\right]  > 0 \text{ for all } t \geq t_{\delta}.$$
This means that to any $\delta > 0$ exists a $t_{\delta} > 0$ such that the set $M_{\delta}$ contains no minimizers of $\mathbb{E}\left[-\ln\left(p(X,[z],t)\right)\right]$ (w.r.t. $[z]$) for $t \geq t_{\delta}$.\\
However, to any $\delta > 0$ exists a $K \in \mathbb{N}$ such that $t_{k} > t_{\delta}$ for all $k \geq K\;(k \in \mathbb{N})$, and any $[y]_{t_{k}} \in E_{t_{k}}(X)$ is by definition of the diffusion $t$-mean sets a minimizer in $\mathbb{RP}^{m}$ of the function
$$[z] \to \mathbb{E}\left[-\ln\left(p(X,[z],t_{k})\right)\right].$$
Therefore, the following statement holds too:
$$\forall \delta > 0 \exists K \in \mathbb{N}\forall k \geq K\;(k \in \mathbb{N}): [y]_{t_{k}} \notin M_{\delta},$$
and this statement is equivalent to the statement
$$\forall \delta > 0 \exists K \in \mathbb{N}\forall k \geq K\;(k \in \mathbb{N}): f([y]_{t_{k}}) \leq \delta,$$
implying convergence of the sequence $(f([y]_{t_{k}}))_{k}$ to $0$.\\
Since $f([y_{\infty}]) = 0$, this means that
$$ \lim_{k \to \infty} f([y]_{t_{k}}) = f([y_{\infty}]).$$
Let $([y]_{t_{l}})_{l}$ be a subsequence of $([y]_{t_{k}})_{k}$. Consider a sequence $(y_{t_{l}})_{l} \subseteq \mathbb{S}^{m}$ of representatives of $[y]_{t_{l}}$. Since $\mathbb{S}^{m}$ is compact, there exists by Bolzano-Weierstraß' Theorem (see \cite{rudin}[p.46]) a convergent (w.r.t. the chordal metric) subsequence of $(y_{t_{l}})_{l}$, denoted $(y_{t_{h}})_{h}$, with limit $\hat{y} \in \mathbb{S}^{m}$. Then
\begin{align*}
\max_{[z] \in \mathbb{RP}^{m}} z^{T} \mathbb{E}[XX^{T}]z -  y_{\infty}^{T}\mathbb{E}[XX^{T}]y_{\infty} &=  f([y_{\infty}]) = \lim_{k \to \infty} f([y]_{t_{k}}) = \lim_{h \to \infty}f([y]_{t_{h}}) \\
&=\lim_{h \to \infty}\left( \max_{[z] \in \mathbb{RP}^{m}} z^{T} \mathbb{E}[XX^{T}]z - (y_{t_{h}})^{T}\mathbb{E}[XX^{T}]y_{t_{h}}\right) \\
&= \max_{[z] \in \mathbb{RP}^{m}} z^{T} \mathbb{E}[XX^{T}]z - \lim_{h \to \infty} y_{t_{h}}^{T}\mathbb{E}[XX^{T}]y_{t_{h}}\\
&= \max_{[z] \in \mathbb{RP}^{m}} z^{T} \mathbb{E}[XX^{T}]z -  \hat{y}\mathbb{E}[XX^{T}]\hat{y},
\end{align*}
from which follows
$$ \hat{y}\mathbb{E}[XX^{T}]\hat{y} = y_{\infty}^{T}\mathbb{E}[XX^{T}]y_{\infty},$$
and therefore $\hat{y} = y_{\infty} \lor -y_{\infty}$, since $[y_{\infty}]$ is the only element in $M_{0}$.
This means that the sequence $(y_{t_{h}})_{h}$ converges to $y_{\infty}$ or to $-y_{\infty}$ with respect to the chordal metric, and this in turn implies that $([y]_{t_{h}})_{h}$ converges with respect to the residual metric to $[y_{\infty}]$. (See Remark~\ref{rmk:chor_to_res})\\
Since the subsequence $([y]_{t_{l}})_{l}$ of $([y]_{t_{k}})_{k}$ was chosen arbitrarily, this means that any subsequence of $([y]_{t_{k}})_{k}$ has a subsequence that converges to $[y_{\infty}]$ with respect to the residual metric, and this statement implies convergence of $([y]_{t_{k}})_{k}$ to $[y_{\infty}]$ with respect to the residual metric.
\end{proof}

\begin{Rm}\label{rmk:chor_to_res}
It was used in the proof of Theorem~\ref{thm:diff_mean} that convergence of a sequence $(x_{k})_{k} \subseteq \mathbb{S}^{m}\; (m \in \mathbb{N}_{\ge 2})$ to $\hat{x} \in \mathbb{S}^{m}$ with respect to the chordal metric implies convergence of the sequence $([x_{k}])_{k} \subseteq \mathbb{RP}^{m}$ to $[\hat{x}] \in \mathbb{RP}^{m}$ with respect to the residual metric. This can be seen by the following argument:\\
The function 
$$f_{\hat{x}}:\mathbb{S}^{m} \to \mathbb{R},\;f_{\hat{x}}(x) := \sqrt{1-(x^T\hat{x})^2}$$
is continuous with respect to the chordal (= Euclidean) metric. In consequence:
$$\lim_{k \to \infty} d_{res}([x_{k}],[\hat{x}]) 
= \lim_{k \to \infty} f_{\hat{x}}(x_{k}) 
= f_{\hat{x}}(\lim_{k \to \infty} x_{k})  
= f_{\hat{x}}(\hat{x}) = 0.$$
\end{Rm}

The result above holds for radius $r=1$ and states that the long time diffusion limit reduces the expected logarithmic heat kernel to $a y^T \mathbb{E}[XX^{T}] y + b$ with constants $a$ and $b$ and thus the long diffusion time limit of the diffusion means can be defined as $ y_{\infty}^T\mathbb{E}[XX^T]y_{\infty} = \max_{[z] \in \mathbb{RP}^{m}}z^T\mathbb{E}[XX^T]z$. Next, we extend the result to arbitrary radius $r$.

\begin{Thm}\cite[Theorem~4.2.7]{Duesberg2024_master}\label{thm:diff_mean_sets_r}
Let $m \in \mathbb{N}_{\ge 2}, r > 0$ and let $X$ be random variable mapping to $\mathbb{RP}^{m}(r)$. Then:
$$E_{t}(X) = rE_{\frac{t}{r^2}}\left(\frac{X}{r}\right),$$
i.e., $E_{t}(X) = \{[y]\in\mathbb{RP}^{m}(r)| \left[\frac{y}{r}\right] \in E_{\frac{t}{r^2}}\left(\frac{X}{r}\right) \}$.
\end{Thm}
\begin{proof}
From
$$p_{\mathbb{RP}^{m}(r)}([x],[y],t) =\frac{1}{r^m} p_{\mathbb{RP}^{m}}\left(\left[\frac{x}{r}\right],\left[\frac{y}{r}\right],\frac{t}{r^2}\right)$$
follows
$$\mathbb{E}[-\ln(p_{\mathbb{RP}^{m}(r)}(X,[y],t))] = -\ln\left(\frac{1}{r^m}\right)+ \mathbb{E}\left[-\ln\left(p_{\mathbb{RP}^{m}}\left(\frac{X}{r},\left[\frac{y}{r}\right],\frac{t}{r^2}\right)\right)\right]$$
and this in turn implies
\begin{align*}
&\argmin_{[y] \in \mathbb{RP}^{m}(r)}\mathbb{E}[-\ln(p_{\mathbb{RP}^{m}(r)}(X,[y],t))] \\
&= \left\{ [y] \in \mathbb{RP}^{m}(r)\big| \left[\frac{y}{r}\right] \in \argmin_{[z] \in \mathbb{RP}^{m}}\mathbb{E}\left[-\ln\left(p_{\mathbb{RP}^{m}}\left(\frac{X}{r},[z],\frac{t}{r^2}\right)\right)\right]\right\}.
\end{align*}
This statement is equivalent to
$$E_{t}(X) = \left\{[y]\in\mathbb{RP}^{m}(r)\big| \left[\frac{y}{r}\right] \in E_{\frac{t}{r^2}}\left(\frac{X}{r}\right) \right\}.$$
\end{proof}

\begin{Thm}\cite[Theorem~4.2.8]{Duesberg2024_master}\label{thm:diff_mean_char}
Let $m \in \mathbb{N}_{\ge 2}, r > 0$ and $X$ be a random variable mapping to $\mathbb{RP}^{m}(r)$ equipped with the residual metric. Then $E^{lower}_{\infty}(X)$ and $E^{upper}_{\infty}(X)$ are contained in $D_{r} / \sim$, where $D_{r}$ denotes the intersection of $\mathbb{S}^{m}(r)$ with the eigenspace of the largest  eigenvalue of $\mathbb{E}[XX^{T}]$, and $$x \sim y :\Leftrightarrow x = y \lor x = -y.$$
%Further, $(E_{t}(X))_{t>0}$ converges to $D_{r}/ \sim $ in the sense of Ziezold.\\
If the largest eigenvalue of $\mathbb{E}[XX^{T}]$ has multiplicity $1$, then 
$$E^{lower}_{\infty}(X) = E^{upper}_{\infty}(X) = D_{r}/\sim. $$
\end{Thm}
\begin{proof}
It will be shown that 
$$E_{\infty}^{lower}(X) = r E_{\infty}^{lower}\left(\frac{X}{r}\right) \text{ and } E_{\infty}^{upper}(X) = r E_{\infty}^{upper}\left(\frac{X}{r}\right),$$
since this statement implies, using the results from Theorem~\ref{thm:diff_mean}, that 
$$E_{\infty}^{lower}(X) = E_{\infty}^{upper}(X) \subseteq r(D/\sim) = D_{r}/\sim,$$
where $D$ denotes the eigenspace to the largest eigenvalue of $\mathbb{E}\left[\frac{X}{r}\left(\frac{X}{r}\right)^T\right]$ intersected with $\mathbb{S}^{m}$, and where equality holds if the largest eigenvalue of $\mathbb{E}[XX^T]$, which is also the largest eigenvalue of $\frac{1}{r^2}\mathbb{E}[XX^T] = \mathbb{E}\left[\frac{X}{r}\left(\frac{X}{r}\right)^T\right]$, has multiplicity $1$.\\

Let $0 < t_{k} \stackrel{k\to \infty}{\to} \infty$ be any monotone increasing sequence, and let the sequence $([y]_{{t}_{k}})_{k}, [y]_{t_{k}} \in E_{t_{k}}(X)$ be such that $([y]_{{t}_{k}})_{k}$ converges to a $[y] \in \mathbb{RP}^{m}(r)$ with respect to the residual metric. The convergence of $([y]_{{t}_{k}})_{k}$ to $[y]$ is equivalent to the convergence of the sequence $(\left[\frac{y}{r}\right]_{{t}_{k}})_{k}$ to $\left[\frac{y}{r}\right]$ in the residual metric on $\mathbb{RP}^{m}$.
Further, by Theorem~\ref{thm:diff_mean_sets_r}: 
$$\left[\frac{y}{r}\right]_{{t}_{k}} \in E_{\frac{t_{k}}{r^2}}\left(\frac{X}{r}\right) \text{ for all } k \in \mathbb{N}.$$ 
In consequence holds for any monotone increasing sequence $0 < t_{k} \stackrel{k\to \infty}{\to} \infty$:
$$\text{Li}_{k \to \infty} E_{t_{k}}(X) =  r \text{Li}_{k \to \infty} E_{\frac{t_{k}}{r^2}}\left(\frac{X}{r}\right)$$
and
$$\text{Ls}_{k \to \infty} E_{t_{k}}(X) =  r \text{Ls}_{k \to \infty} E_{\frac{t_{k}}{r^2}}\left(\frac{X}{r}\right),$$
yielding
$$ \bigcap_{t_{k} \stackrel{k\to \infty}{\to}  \infty} \text{Li}_{k \to \infty} E_{t_{k}}(X) =  \bigcap_{t_{k} \stackrel{k\to \infty}{\to} \infty} r \text{Li}_{k \to  \infty} E_{\frac{t_{k}}{r^2}}\left(\frac{X}{r}\right) = r \bigcap_{\tilde{t}_{k} \stackrel{k\to \infty}{\to} \infty} \text{Li}_{k \to \infty} E_{\tilde{t}_{k}}\left(\frac{X}{r}\right)  $$
as well as
$$
\bigcap_{t_{k} \stackrel{k\to \infty}{\to} \infty} \text{Ls}_{k \to \infty} E_{t_{k}}(X) =  \bigcap_{t_{k} \stackrel{k\to \infty}{\to} \infty} r\text{Ls}_{k \to \infty} E_{\frac{t_{k}}{r^2}}\left(\frac{X}{r}\right) = r \bigcap_{\tilde{t}_{k} \stackrel{k\to \infty}{\to}  \infty} \text{Ls}_{k \to \infty} E_{\tilde{t}_{k}}\left(\frac{X}{r}\right) .
$$
This shows, that 
$$E_{\infty}^{lower}(X) = r E_{\infty}^{lower}\left(\frac{X}{r}\right)$$ 
and 
$$ E_{\infty}^{upper}(X) = r E_{\infty}^{upper}\left(\frac{X}{r}\right).$$
\end{proof}

We have now fully described the long time limit of the diffusion mean set. Next, we investigate the relation of this set to the extrinsic mean set in the smooth isotropic embedding of $\mathbb{RP}^{m}$ into a Euclidean space. For this, we use an explicit construction of the isometric embedding into a Euclidean space of minimal dimension stated recursively in dimension $m$ by \cite{zhang}.

\begin{Thm}\cite{zhang}\label{thm:zhang}
Let  $m \in \mathbb{N}_{\ge 2}$ and $r_{m} = \sqrt{\frac{2(m+1)}{m}}$.
Then the real projective space $\mathbb{RP}^{m}(r_{m})$ %equipped with the standard Euclidean metric on the sphere 
%\mathbb{S}^{m}\left(\sqrt{\frac{2(n+1)}{n}}\right)$ 
can be embedded 
isometrically in $\mathbb{S}^{\frac{m(m+3)}{2}-1}$ with the embedding $\Phi$ given by
$$ \Phi([x]) := \frac{1}{2}\sqrt{\frac{m}{2(m+1)}}F_{m}(x,x),$$
where $F_{m}:\mathbb{R}^{m+1} \times \mathbb{R}^{m+1} \to \mathbb{R}^{\frac{m(m+3)}{2}}$ is recursively defined by 
\begin{align*}
F_{1}((x_{1},x_{2}),(y_{1},y_{2})) &= (x_{1}y_{1}-x_{2}y_{2}, x_{1}y_{2}+x_{2}y_{1}) \\
F_{m+1}\big((x,x_{m+2}),(y,y_{m+2})\big)&=
\big(F_{m}(x,y),x_{m+2}y + y_{m+2}x, \tau_{m+1}(\langle x,y \rangle - (m+1) x_{m+2}y_{m+2})\big)
\end{align*}
where $\tau_{m+1} = \sqrt{\frac{2}{(m+1)(m+2)}}$.
\end{Thm}

Since the long time limit of the diffusion mean set is defined by $\mathbb{E}[\langle X,y \rangle^{2}]$, we set out to show that this term is affinely related to the squared chordal distance in the isometric embedding. This result may seem surprising at first glance, because $\langle X,y \rangle^{2}$ appears to be related to the fourth order of chordal distance in the non-isometric embedding derived from the isometric embedding of the sphere. However, if one notes that the isometric embedding described by \cite{zhang} employs the Veronese map of order 2, the result intuitively makes sense.

\begin{Lem}\cite[Lemma~4.3.2]{Duesberg2024_master}\label{lem:auxi}
Let $m \in \mathbb{N}_{\ge 2}, r_{m} =  \sqrt{\frac{2(m+1)}{m}}$ and $[x],[y] \in \mathbb{RP}^{m}\left(r_{m}\right)$.
Then 
$$\frac{2}{m(m+1)}\left(m^2||x||^4 + m^2||y||^4 +2m||x||^2||y||^2\right) = 4\left(\frac{2(m+1)}{m}\right)^2.$$
\end{Lem}
\begin{proof}
\begin{align*}
\frac{2}{m(m+1)}\left(m^2||x||^4 + m^2||y||^4 +2m||x||^2||y||^2\right) &= \frac{2}{m(m+1)}\left(2m^2r_{m}^4 +2mr_{m}^4\right)\\
&= 4 r_{m}^4 = 4 \left(\frac{2(m+1)}{m}\right)^2.
\end{align*}
\end{proof}

\begin{Thm}\cite[Theorem~4.3.3]{Duesberg2024_master}\label{thm:embed_rel}
Let  $m \in \mathbb{N}_{\ge 2}$ and $r_{m} = \sqrt{\frac{2(m+1)}{m}}$. Let $[x],[y] \in \mathbb{RP}^{m}\left(r_{m}\right)$, and $\Phi$ the isometric embedding from Theorem~\ref{thm:zhang}. Then the following equality holds:
$$||\Phi([x]) - \Phi([y])||^2 = -\frac{m}{2(m+1)}\langle x,y \rangle^2 + \frac{2(m+1)}{m}.$$
In consequence, 
$$\argmax_{[y] \in \mathbb{RP}^{m}(r_{m})} \mathbb{E}[\langle X,y \rangle^{2}] = \argmin_{[y] \in \mathbb{RP}^{m}\left(r_{m}\right)} \mathbb{E}\left[||\Phi([X])-\Phi([y])||^2\right].$$
\end{Thm}
\begin{proof}
The Theorem will be proven by induction over $m$.\\
Since 
$$||\Phi([x]) - \Phi([y])||^2 = \frac{1}{4}\frac{m}{2(m+1)}||F_{m}(x,x) - F_{m}(y,y)||^2,$$
it suffices to show that 
$$ ||F_{m}(x,x) - F_{m}(y,y)||^2 = -4 \left(\langle x,y \rangle^2 - \left(\frac{2(m+1)}{m}\right)^2\right)$$
for all $m \in \mathbb{N}_{\ge 2}$.\\

\textbf{Base case}:\\
Consider the case of $m=1$. 
Then for $[(x_{1},x_{2})], [(y_{1},y_{2})] \in \mathbb{RP}^{1}\left(\sqrt{\frac{2(1+1)}{1}}\right)$:
\begin{align*}
&||F_{1}((x_{1},x_{2}),(x_{1},x_{2}))- F_{1}((y_{1},y_{2}),(y_{1},y_{2}))||^2\\
&= ||(x_{1}^2-x_{2}^2,2x_{1}x_{2}) - (y_{1}^2-y_{2}^2,2y_{1}y_{2})||^2 \\
&= ||\left((x_{1}^2 - x_{2}^2)-(y_{1}^2 - y_{2}^2), 2(x_{1}x_{2} - y_{1}y_{2})\right)||^2\\
&= (x_{1}^2-x_{2}^2)^2 + (y_{1}^2-y_{2}^2)^2 - 2(x_{1}^2-x_{2}^2)(y_{1}^2-y_{2}^2) + 4 (x_{1}x_{2}-y_{1}y_{2})^2\\
&= -4 (x_{1}^2y_{1}^2 + x_{2}^2y_{2}^2 +2 x_{1}y_{1}x_{2}y_{2}) 
+ 2 (x_{1}^2y_{1}^2 + x_{2}^2y_{2}^2) 
+ 2 (x_{1}^2y_{2}^2 + x_{2}^2y_{1}^2)
+ 4 (x_{1}^2x_{2}^2 + y_{1}^2y_{2}^2)\\ 
&\;\;\;\;+ (x_{1}^2-x_{2}^2)^2 + (y_{1}^2-y_{2}^2)^2 \\
&= -4 \langle (x_{1},x_{2}),(y_{1},y_{2}) \rangle^2 +
(x_{1}^4 + x_{2}^4 + 2x_{1}^2x_{2}^2) +
(y_{1}^4 + y_{2}^4 + 2y_{1}^2y_{2}^2) +
2 (x_{1}^2 + x_{2}^2)(y_{1}^2 + y_{2}^2) \\
&= -4 \langle (x_{1},x_{2}),(y_{1},y_{2}) \rangle^2 +
\frac{2}{1(1+1)} \left(1^2||(x_{1},x_{2})||^4 + 1^2||(y_{1},y_{2})||^4 + 2(1)||(x_{1},x_{2})|^2||(y_{1},y_{2})||^2 \right)\\
&=^{\text{Lemma~\ref{lem:auxi}}} -4 \left( \left\langle (x_{1},x_{2}), (y_{1},y_{2}) \right\rangle^2 - \left(\frac{2(1+1)}{1}\right)^2\right).
\end{align*}
\newpage
\textbf{Induction step:}\\
Now suppose the statement is true for an $m\in\mathbb{N}_{\ge 2}$. Then for $[(x,x_{m+2})], [(y,y_{m+2})] \in \mathbb{RP}^{m+1}\left(\sqrt{\frac{2(m+2)}{m+1}}\right)$:
\begin{align*}
&||F_{m+1}((x,x_{m+2}),(x,x_{m+2}))- F_{m+1}((y,y_{m+2}),(y,y_{m+2}))||^2\\
&= ||F_{m}(x,x)-F_{m}(y,y)||^2 + 4||x_{m+2}x - y_{m+2}y||^2 + \\
&\frac{2}{(m+1)(m+2)} \Big((\langle x,x \rangle - \langle y,y \rangle) - (m+1)(x_{m+2}^2 - y_{m+2}^2)\Big)^2 \\
&=^{\text{Lemma~\ref{lem:auxi}}} - 4 \langle x,y \rangle^2 + \frac{2}{m(m+1)}\left( m^2||x||^4 + m^2||y||^4 + 2m||x||^2||y||^2\right) \\
&+ 4 \sum_{i=1}^{m+1}(x_{m+2}x_{i} - y_{m+2}y_{i})^2 
+ \frac{2}{(m+1)(m+2)}\big((||x||^2 - ||y||^2) - (m+1)(x_{m+2}^2 - y_{m+2}^2)\big)^2 \\
&= -4 \langle x,y \rangle^2 -4 x_{m+2}^2y_{m+2}^2 - 8 x_{m+2}y_{m+2} \langle x,y \rangle \\
&+ \frac{2}{(m+1)(m+2)}\Big(((m+2)m+1)||x||^4 
+ ((m+2)m+1)||y||^4 
+ (2(m+2)-2)||x||^2||y||^2\\
&+ (m+1)^2 x_{m+2}^4 + (m+1)^2 y_{m+2}^4
+ (2(m+1)(m+2)- 2(m+1)^2)x_{m+2}^2y_{m+2}^2 \\
&+ (2(m+1)(m+2)-2(m+1))||x||^2x_{m+2}^2 
+ (2(m+1)(m+2)-2(m+1))||y||^2y_{m+2}^2\\
&+ 2(m+1)||x||^2y_{m+2}^2 
+ 2(m+1)||y||^2x_{m+2}^2\Big)\\
&= -4 \langle (x,x_{m+2}),(y,y_{m+2}) \rangle^2 
+ \frac{2}{(m+1)(m+2)} \Big((m+1)^2(||x||^4 + 2||x||^2x_{m+2}^2 + 
x_{m+2}^4) \\
&+(m+1)^2(||y||^4 + 2||y||^2y_{m+2}^2 + y_{m+2}^4) 
+ 2(m+1)( (||x||^2 + x_{m+2}^2)(||y||^2 + y_{m+2}^2))\Big) \\
&= -4 \langle (x,x_{m+2}),(y,y_{m+2}) \rangle^2 +\frac{2}{(m+1)(m+2)}\Big( (m+1)^2||(x,x_{m+2})||^4 + (m+1)^2||(y,y_{m+2})||^4 \\
&+2(m+1)||(x,x_{m+2})||^2|(y,y_{m+2})||^2\Big)\\
%&= -4 \langle (x,x_{m+2}),(y,y_{m+2}) \rangle^2 
% + \frac{2}{(m+1)(m+2)}\left(((m+1)^2 + (m+1)^2 + 2(m+1)) \left(\frac{2(m+2)}{m+1}\right)^2 \right) \\
%&= -4 \langle (x,x_{m+2}),(y,y_{m+2}) \rangle^2 +
%\frac{2}{(m+2)}(2(m+1)+2)\left(\frac{2(m+2)}{m+1}\right)^2\\
&=^{\text{Lemma~\ref{lem:auxi}}} -4 \left( \left\langle (x,x_{m+2}),(y,y_{m+2}) \right\rangle^2 -
\left(\frac{2(m+2)}{m+1}\right)^2\right).
\end{align*}
\end{proof}

Now, we have all tools to show that the long time limit of the diffusion means is a subset of the extrinsic mean set in the isometric embedding of real projective spaces of radius $r_{m} = \sqrt{\frac{2(m+1)}{m}}$. The first part of the following theorem states this result in terms of the definition of the extrinsic mean as minimizer of expected squared chordal distance. The second part of the theorem relates this to the more well-known definition the extrinsic mean on a sphere, leveraging the fact that $\mathbb{RP}^m(r_m)$ is embedded into a unit sphere in the isometric embedding.

\begin{Thm}\cite[Theorem~4.3.4]{Duesberg2024_master}\label{thm:results_rp_embed}
Let  $m \in \mathbb{N}_{\ge 2}$,$r_{m} = \sqrt{\frac{2(m+1)}{m}}$ and $\Phi$ as in Theorem~\ref{thm:zhang}. Let $X$ be a random variable mapping to $\mathbb{RP}^{m}(r_{m})$.
Then:
\begin{enumerate}
\item[1)] $E_{\infty}^{lower}(X)$ and $E_{\infty}^{upper}(X)$ are contained in 
$$ M := \Phi^{-1}\left(\argmin_{z \in \Phi(\mathbb{RP}^{m}(r_{m}))} \mathbb{E}[||\Phi(X)-z||^2]\right),$$
and are equal to $M$, if the largest eigenvalue of $\mathbb{E}[XX^{T}]$ has algebraic/geometric multiplicity $1$. 
\item[2)] If the largest eigenvalue of $\mathbb{E}[XX^{T}]$ has multiplicity $1$ and if additionally $\mathbb{E}[\Phi(X)] \neq 0$ and  $\frac{\mathbb{E}[\Phi(X)]}{||\mathbb{E}[\Phi(X)]||} \in \Phi\left(\mathbb{RP}^{m}\left(r_{m}\right)\right)$, 
$E^{lower}_{\infty}(X)$ and $E^{upper}_{\infty}(X)$ coincide with the extrinsic mean w.r.t. $\Phi$ of $X$ given by   
$$ [y]_{\infty} := \Phi^{-1}\left(\frac{\mathbb{E}[\Phi(X)]}{||\mathbb{E}[\Phi(X)]||}\right).$$
\end{enumerate}
\end{Thm}
\begin{proof}
From Theorem~\ref{thm:diff_mean_char} it is known that $E_{\infty}^{lower}(X)$ and $E_{\infty}^{upper}(X)$ are contained in $D_{r_{m}} / \sim$ where $D_{r_{m}}$ denotes the intersection of $\mathbb{S}^{m}\left(r_{m}\right)$
with the eigenspace of the largest eigenvalue of $\mathbb{E}[XX^{T}]$, and that $E_{\infty}^{lower}(X)$ and $E_{\infty}^{upper}(X)$ are equal to $D_{r_{m}} / \sim$, if the largest eigenvalue of $\mathbb{E}[XX^T]$ has multiplicity $1$. \\
In consequence, to show the first statement of the Theorem, it suffices to show that  $M = D_{r_{m}} / \sim.$\\
Since 
$$\mathbb{E}[\langle X,y \rangle^2] = \mathbb{E}[(X^{T}y)^{T}X^{T}y] = \mathbb{E}[y^{T}XX^{T}y] = y^{T}\mathbb{E}[XX^{T}]y,$$
the maximizers of $\mathbb{E}[\langle X,y \rangle^2]$ in $\mathbb{RP}^{m}(r_{m})$ are exactly the elements of $D_{r_{m}} / \sim$. 
Therefore:
$$ \argmax_{[y] \in \mathbb{RP}^{m}(r_{m})} \mathbb{E}[\langle X,y \rangle^2] = D_{r_{m}} / \sim .$$
By Theorem~\ref{thm:embed_rel}, 
$$M = \argmax_{[y] \in \mathbb{RP}^{m}(r_{m})} \mathbb{E}[\langle X,y \rangle^2],$$
and therefore
$$ M = D_{r_{m}} / \sim.$$
For the second statement of the Theorem, note that for any $z \in \mathbb{S}^{\frac{m(m+3)}{2}-1}$
\begin{align*}
\mathbb{E}[||\Phi(X)-z||^2] &= \mathbb{E}[\langle \Phi(X)-z, \Phi(X)-z \rangle] \\
&= \mathbb{E}[\langle \Phi(X),\Phi(X) \rangle + \langle z,z \rangle - 2\langle  \Phi(X),z \rangle]\\
&= \mathbb{E}[2 - 2\langle  \Phi(X),z \rangle] \\
&= 2 - 2\langle \mathbb{E}[\Phi(X)],z \rangle.
%&= 2(1-||\mathbb{E}[\Phi(X)]||\cos(\sphericalangle(\mathbb{E}[\Phi(X)],z)))
\end{align*}
Therefore, if $\mathbb{E}[\Phi(X)] \neq 0$, then 
$\mathbb{E}[||\Phi(X)-z||^2]$
is minimized with respect to $ z \in \mathbb{S}^{\frac{m(m+3)}{2}-1}$ uniquely by $\frac{\mathbb{E}[\Phi(X)]}{||\mathbb{E}[\Phi(X)]||}$.
In consequence, if $\frac{\mathbb{E}[\Phi(X)]}{||\mathbb{E}[\Phi(X)]||}$ is contained in $\Phi\left(\mathbb{RP}^{m}(r_{m})\right)$, noting that as an isometric embedding $\Phi$ is an injective function, $M$ consists only of the element
$$[y]_{\infty} := \Phi^{-1}\left(\frac{\mathbb{E}[\Phi(X)]}{||\mathbb{E}[\Phi(X)]||}\right).$$
This element $[y]_{\infty}$ is also the extrinsic mean with respect to $\Phi$ of $X$, since $\frac{\mathbb{E}[\Phi(X)]}{||\mathbb{E}[\Phi(X)]||}$ minimizes uniquely the expression $||y- \mathbb{E}[\Phi(X)]||$ with respect to $y \in \Phi(\mathbb{RP}^{m}(r))$.
\end{proof}

Finally, we would like to extend the result of Theorem~\ref{thm:results_rp_embed} to arbitrary $r > 0$. First we note that the definition of the map $\Phi$ via the map $F_m$ immediately permits an extension of $\Phi$ to a map $\widetilde{\Phi} :\mathbb{R}^{m+1} \to \mathbb{R}^{\frac{m(m+3)}{2}}$. With this map we can now state a scaling result.

\begin{Lem}
    For every $r \in \mathbb{R}_{>0}$ and $x \in \mathbb{R}^{m+1}$,
    \[\widetilde{\Phi}(rx) = r^2 \widetilde{\Phi}(x)\]
\end{Lem}

\begin{proof}
    We show by induction that for every $ m \ge 1$, $F_m(rx,ry) = r^2 F_m(x,y)$.
    \begin{align*}
        F_{1}((r x_{1},r x_{2}),(r y_{1},r y_{2})) =&~ (r x_{1}r y_{1}-r x_{2}r y_{2}, r x_{1} r y_{2}+ r x_{2}r y_{1})\\
        =&~ r^2 (x_{1} y_{1}- x_{2} y_{2}, x_{1} y_{2} + x_{2} y_{1}) = r^2 F_{1}((x_{1}, x_{2}),(y_{1}, y_{2})) \\
        F_{m+1}\big((r x,r x_{m+2}),(r y,r y_{m+2})\big)= &~
        \big(F_{m}(r x,r y),r x_{m+2} r y + r y_{m+2} rx,\\
        &~~\tau_{m+1}(\langle r x,r y \rangle - (m+1) r x_{m+2}r y_{m+2})\big)\\
        =&~ r^2
        \big(F_{m}(x,y),x_{m+2}y + y_{m+2}x,\\
        &~~\tau_{m+1}(\langle x,y \rangle - (m+1) x_{m+2}y_{m+2})\big)\\
        =&~ r^2 F_{m+1}\big((x, x_{m+2}),(y, y_{m+2})\big)
    \end{align*}
    The claim follows at once.
\end{proof}

In consequence, it is immediately clear that $\Phi_r := \frac{r_m}{r}\widetilde{\Phi}|_{\mathbb{RP}^m(r)}$ defines a map $\Phi_r : \mathbb{RP}^m(r) \to \mathbb{S}^{\frac{m(m+1)}{2}}\left(\frac{r}{r_m}\right) \subset \mathbb{R}^{\frac{m(m+3)}{2}}$.

\begin{Thm}
    For any $[x] \in \mathbb{RP}^m(r_m)$, this map satisfies the relation $\Phi_{r}\left(\frac{r}{r_{m}}[x]\right) = \frac{r}{r_{m}}\Phi\left([x]\right)$ and is an isometric embedding.
\end{Thm}
\begin{proof}
    This follows from the fact that $\Phi$ is an isometric embedding: All tangent vectors are thus scaled by $\frac{r}{r_{m}}$, in both the domain as well as the image, thus the isometry property follows immediately.
\end{proof}

With this last result, we have extended the result to real projective spaces of arbitrary radii, thus making good on the conjecture to this effect in \cite[Remark~4.3.5]{Duesberg2024_master}

\section{Outlook}
\label{sec:outlook}
We have shown that the diffusion mean set on real projective spaces, endowed with the canonical metric tensor inherited from the sphere, converges to a subset of the extrinsic mean set in the smooth isometric embedding. This extends the prior result for the circle and spheres of arbitrary dimension. From these results we derive the conjecture

\begin{Conj} \label{conj:diff-ex}
  For every connected compact symmetric space, whose smooth isometric embedding into a Euclidean space of minimal dimension is unique up to isometries of the Euclidean space, the diffusion mean set converges to the extrinsic mean set in this isometric embedding in the sense of the upper and lower Kuratowski limit.
\end{Conj}

Based on this conjecture, we introduce the shorthand \emph{isometric-extrinsic mean set} for the \emph{extrinsic mean set in the smooth isometric embedding of lowest embedding space dimension unique up to isometries of the embedding space}. If the conjecture holds true, the \emph{isometric-extrinsic mean set} becomes a viable location statistic on connected compact symmetric spaces and its identification with the long time diffusion limit makes it easily calculable without actually calculating the isometric embedding, which can be very tedious. Instead, the long time limit of the heat kernel reduces it to a term defined by the eigenfunction of the Laplace-Beltrami-Operator with the second largest eigenvalue, which can usually be easily optimized.

It is impractical to approach this conjecture one case at a time. Instead it is highly desirable to leverage general results from differential geometry to attempt a proof to Conjecture~\ref{conj:diff-ex}. If we denote by $\psi^{\mathcal{M}}_1$ the eigenfunction to the second largest eigenvalue of the Laplace-Beltrami operator on a manifold $\mathcal{M}$, we can formulate the following possible avenue towards a proof to may be to Conjecture~\ref{conj:diff-ex}: check whether the isometric embedding $\Phi$ of a suitable compact symmetric space $\mathcal{M}$ always embeds it into a sphere $\mathbb{S}^M$ of some dimension $M$ in such a way that $\psi^{\mathbb{S}^M}_1$ remains an eigenfunction under restriction of the Laplace-Beltrami operator to the embedded $\Phi(\mathcal{M})$ and is affinely related to the push-forward $\Phi^*\psi^{\mathcal{M}}_1$. If this can be shown, Conjecture~\ref{conj:diff-ex} follows immediately.

\section*{Acknowledgments}

B. Eltzner gratefully acknowledges funding by the DFG~CRC~1456 project~B02 as well as travel funding by the organizers of the HeKKSaGOn workshops ``Analysis, Geometry and Stochastics
on Metric Spaces'' and ``Metric and Measures''.

\bibliographystyle{Chicago}
\bibliography{BibPaper}

\end{document}